\documentclass{amsart}
\usepackage{amsmath}
\usepackage{amsfonts}
\usepackage{amssymb}
\usepackage{graphicx}
\usepackage{amsthm}
\usepackage{mathrsfs}
\usepackage{pgf,tikz}
\usepackage{multirow}
\usepackage[fleqn]{mathtools}
\usepackage{bm}

\def \Cay {\mathrm{Cay}}

\def \cC {\mathcal C}

\def \cF {{\mathcal F}}

\def \cS {{\mathcal S}}
\def \cT {{\mathcal T}}

\def \Z  {\mathbb{Z}}

\def \G  {\Gamma}
\def \ov {\overline}

\newtheorem{defi}{Definition}[section]
\newtheorem{prop}[defi]{Proposition}
\newtheorem{lem}[defi]{Lemma}

\newtheorem{cor}[defi]{Corollary}
\newtheorem{thm}[defi]{Theorem}
\newtheorem{rem}[defi]{Remark}

\newtheorem{conj}{Conjecture}
\theoremstyle{definition}
\newtheorem{ex}[defi]{Example}

\newcommand{\rnote}[1]{{\color{red}{\sf #1}}}
\newcommand{\bnote}[1]{{\color{blue}{\sf #1}}}

\newcommand{\comment}[1]{}

\begin{document}
\title{A reduction of the spectrum problem for odd sun systems and the prime case}

\author{Marco Buratti}
\address{Dipartimento di Matematica e Informatica, Universit\`a di Perugia, Via Vanvitelli 1, I-06123 Perugia, Italy}
\email{buratti@dmi.unipg.it}

\author{Anita Pasotti}
\address{DICATAM - Sez. Matematica, Universit\`a degli Studi di Brescia, Via
Branze 43, I-25123 Brescia, Italy}
\email{anita.pasotti@unibs.it}

\author{Tommaso Traetta}
\address{DICATAM - Sez. Matematica, Universit\`a degli Studi di Brescia, Via
Branze 43, I-25123 Brescia, Italy}
\email{tommaso.traetta@unibs.it}

\setcounter{MaxMatrixCols}{18}

\begin{abstract}
  A $k$-cycle with a pendant edge attached to each vertex is called a $k$-sun.
  The existence problem for $k$-sun decompositions of $K_v$, with $k$ odd, has been solved only when
  $k=3$ or $5$.


   By adapting a method used by Hoffmann, Lindner and Rodger to reduce the spectrum problem for odd cycle systems of the complete graph,
  we show that if there is a $k$-sun system of $K_v$ ($k$ odd) whenever $v$ lies in the range $2k< v < 6k$ and
  satisfies the obvious necessary conditions, then such a system exists for every admissible $v\geq 6k$.


  Furthermore, we give a complete solution whenever $k$ is an odd prime.
\end{abstract}
\maketitle
\section{Introduction}
We denote by $V(\G)$ and $E(\G)$ the set of vertices and the list of edges of a graph $\G$, respectively.
Also, we denote by $\G+w$ the graph obtained by adding to $\G$ an independent set
$W=\{\infty_i\mid 1\leq{i}\leq{w}\}$  of $w\geq 0$
vertices each adjacent to every vertex of $\G$, namely,
\[
\G+w := \G\ \cup\ K_{V(\G), W},
\]
where
$K_{V(\G), W}$ is the complete bipartite graph with parts $V(\G)$ and $W$. Denoting by $K_v$ the
\emph{complete graph} of order $v$, it is clear that $K_{v} + 1$ is isomorphic to $K_{v+1}$.

We denote by $x_1 \sim x_2 \sim \ldots \sim x_k$ the \emph{path} with edges $\{x_{i-1}, x_{i}\}$
for $2\leq{i}\leq{k}$. By adding the edge $\{x_1, x_k\}$ when $k\geq 3$, we obtain a \emph{cycle of length $k$}
(briefly, a \emph{$k$-cycle})
denoted by $(x_1, x_2, \ldots, x_k)$. A $k$-cycle with further $v-k\geq0$ isolated vertices will be referred to as
a $k$-cycle of order $v$. By adding to  $(x_1, x_2, \ldots, x_k)$ an independent set of edges
$\big\{\{x_i, x'_i\}\mid 1\leq{i}\leq{k}\big\}$, we obtain the \emph{$k$-sun} on $2k$ vertices
(sometimes referred to as \emph{$k$-crown} graph) denoted by
\begin{equation}\nonumber
  \left(
  \begin{array}{ccccc}
    x_1  & x_2      & \ldots & x_{k-1}   & x_{k} \\
    x'_1 & x'_2     & \ldots & x'_{k-1}  & x'_{k}
  \end{array}
  \right),
\end{equation}
whose edge-set is therefore $\big\{\{x_i, x_{i+1}\}, \{x_i, x'_i\}\mid 1\leq{i}\leq{k}\big\}$,
where $x_{k+1}=x_1$.

A \emph{decomposition} of a graph $K$ is a set $\{\G_1, \G_2, \ldots, \G_t\}$ of subgraphs of $K$
whose edge-sets between them partition the edge-set of $K$; in this case, we briefly write
$K=\oplus_{i=1}^t \G_{i}$. If each $\G_i$ is isomorphic to $\G$, we speak of a
\emph{$\G$-decomposition of $K$}. If $\G$ is a $k$-cycle (resp., $k$-sun), we also speak of a
\emph{$k$-cycle system} (resp., \emph{$k$-sun system}) of $K$.

In this paper we study the existence problem for $k$-sun systems of $K_v$ ($v>1$).
Clearly, for such a system to exist we must have
\begin{equation}\label{nec}\tag{$\ast$}
  \text{$v\geq 2k$\;\;\; and\;\;\; $v(v-1)\equiv 0 \pmod{4k}$}.
\end{equation}
As far as we know, this problem has been completely settled only
when $k=3,5$ \cite{FHL, FJLS}, $k=4,6,8$ \cite{LG}, and when $k=10, 14$ or $2^t\geq 4$ \cite{FJLSdm}.
It is important to notice that, as a consequence of a general result proved in \cite{W},
condition \eqref{nec} is sufficient whenever $v$ is large enough with respect to $k$.
These results seem to suggest the following.
\begin{conj}\label{conj}
Let $k\geq 3$ and $v> 1$.
There exists a $k$-sun system of $K_v$ if and only if \eqref{nec} holds.
\end{conj}
Our constructions rely on the existence of $k$-cycle systems of $K_v$,
a problem that has been completely settled in
\cite{AG, BHS, B2003, HLR, S}. More precisely,
 \cite{BHS} and \cite{HLR}
reduce the problem to the orders $v$ in the range $k\leq v < 3k$, with $v$ odd. These cases are then
solved in \cite{AG, S}. For odd $k$, an alternative proof based on $1$-rotational constructions
is given in \cite{B2003}. Further results on $k$-cycle systems of $K_v$ with an automorphism group
acting sharply transitively on all but at most one vertex can be found in \cite{BP, B2004, BDF, WB}.


The main results of this paper focus on the case where $k$ is odd.
By adapting a method used in \cite{HLR} to reduce the spectrum problem for odd cycle systems of the complete graph,
we show that if there is a $k$-sun system of $K_v$ ($k$ odd) whenever $v$ lies in the range $2k< v < 6k$ and
satisfies the obvious necessary conditions, then such a system exists for every admissible $v\geq 6k$.
In other words, we show the following.

\begin{thm}\label{main1}
Let $k\geq 3$ be an odd integer and $v>1$. Conjecture \ref{conj} is true if and only if there exists a $k$-sun system of $K_v$ for all $v$ satisfying the necessary conditions in \eqref{nec} with $2k< v <6k$.
\end{thm}

We would like to point out that we strongly believe the reduction methods used in
\cite{BHS, HLR} could be further developed to reduce the spectrum problem of other types of graph decompositions of $K_v$.

In Section 6, we construct $k$-sun systems of $K_v$ for every odd prime $k$ whenever $2k< v <6k$ and \eqref{nec} holds. Therefore, as a consequence of Theorem \ref{main1}, we solve the existence problem for $k$-sun systems of $K_v$ whenever $k$ is an odd prime.
\begin{thm}\label{main2}
  For every odd prime $p$ there exists a $p$-sun system of $K_v$ with $v>1$ if and only if
  $v\geq 2p$ and $v(v-1)\equiv 0\pmod{4p}$.
\end{thm}
Both results rely on the difference methods described in Section 2. These methods are used in
Section 3 to construct specific $k$-cycle decompositions of some subgraphs of $K_{2k}+w$,
which we then use in Section 4 to build $k$-sun systems of $K_{4k} + n$. This is the last ingredient we need in Section 5 to prove Theorem \ref{main1}.
Difference methods are finally used in Section 6 to construct $k$-sun systems of $K_v$ for every odd prime $k$ whenever $2k< v <6k$ and \eqref{nec} holds.

\section{Preliminaries}
Henceforward, $k\geq3$ is an odd integer, and $\ell=\frac{k-1}{2}$.
Also, given two integers $a\leq b$, we denote by $[a,b]$ the interval containing the integers $\{a, a+1, \ldots, b\}$.
If $a>b$, then $[a,b]$ is empty.

In our constructions we make extensive use of the method of partial mixed differences
which we now recall but limited to the scope of this paper.

Let $G$ be an abelian group of odd order $n$ in additive notation, let $W=\{\infty_u\mid 1\leq{u}\leq{w}\}$, and denote by $\G$ a graph with vertices in $V=(G\times [0, m-1]) \ \cup\ W$.
For any permutation $f$ of $V$, we denote
by $f(\Gamma)$ the graph obtained by replacing each vertex of $\Gamma$, say $x$, with $f(x)$.
Letting $\tau_g$, with $g\in G$, be
the permutation of $V$ fixing each $\infty_u\in W$ and mapping $(x,i)\in G\times [0,m-1]$ to $(x+g,i)$,
we call $\tau_g$ the \emph{translation by $g$} and $\tau_g(\G)$ the related translate of $\G$.

We denote by $Orb_{G}(\G) = \{\tau_g(\G)\mid g\in G\}$
the \emph{$G$-orbit} of $\G$, that is, the set of all distinct translates of $\G$,
and by $Dev_G(\G)=\bigcup_{g\in G} \tau_g(\G)$ the graph union of all translates of $\G$.
Further, by $Stab_{G}(\G)=\{g\in G \mid \tau_g(\G)=\G \}$ we denote
the \emph{$G$-stabilizer} of $\G$, namely, the set of translations fixing $\G$. We recall that $Stab_G(\G)$ is a subgroup of $G$,
hence $s=|Stab_G(\G)|$ is a divisor of $n=|G|$.
Henceforward, when $G=\Z_k$, we will simply write $Orb(\G)$, $Dev(\G)$, and $Stab(\G)$.

Suppose now that $\G$ is either a $k$-cycle or a $k$-sun with vertices in $V$.
%
%
For every $i,j\in [0,m-1]$, the list of $(i,j)$-differences of $\G$ is the multiset $\Delta_{ij} \G$
defined  as follows:
\begin{enumerate}
  \item if $\G=(x_1, x_2, \ldots, x_k)$, then
  \begin{align*}
    \Delta_{ij} \G & =
    \big\{a_{h+1}-a_h\mid x_h= (a_h,i), x_{h+1}= (a_{h+1},j), 1\leq{h}\leq{k/s} \big\} \\
     & \cup\
    \big\{a_h-a_{h+1}\mid x_h= (a_h,j), x_{h+1}= (a_{h+1},i), 1\leq{h}\leq{k/s} \big\};
  \end{align*}
  \item if
  $\G=\left(
  \begin{array}{cccc}
    x_1  & x_2      & \ldots  & x_{k} \\
    x'_1 & x'_2     & \ldots  & x'_{k}
  \end{array}
  \right)
  $, then
  \begin{align*}
    \Delta_{ij} \G &= \Delta_{ij}(x_1, x_2, \ldots, x_k)\ \cup\
    \big\{a'_h-a_{h}\mid x_h= (a_h,i), x'_{h}= (a'_{h},j), 1\leq{h}\leq{k/s} \big\} \\
    & \cup\
    \big\{a_h-a'_{h}\mid x_h= (a_h,j), x'_{h}= (a'_{h},i), 1\leq{h}\leq{k/s} \big\}.
  \end{align*}
\end{enumerate}
We notice that when $s=1$ we find the classic concept of list of differences.
Usually, one speaks of \emph{pure or mixed differences} according to whether $i=j$ or not, and when
$m=1$ we simply write $\Delta \G$. This concept naturally extends to a family $\mathcal{F}$ of graphs with vertices
in $V$ by setting $\Delta_{ij} \mathcal{F} = \bigcup_{\G\in\cF} \Delta_{ij}\G$.
Clearly, $\Delta_{ij}\G = -\Delta_{ji}\G$, hence $\Delta_{ij}\cF = -\Delta_{ji}\cF$, for
every $i,j\in [0,m-1]$.

We also need to define the \emph{list of neighbours} of $\infty_u$ in $\cF$, that is,
the multiset $N_{\cF}(\infty_u)$ of the vertices in $V$ adjacent to $\infty_u$
in some graph $\Gamma\in\cF$.

Finally, we introduce a special class of subgraphs of $K_{mn}$. To this purpose,
we take $V(K_{mn})=G \times [0,m-1]$. Letting $D_{ii}\subseteq G\setminus\{0\}$
for every $0\leq i\leq m-1$, and $D_{ij}\subseteq G$ for every $0\leq i< j\leq m-1$,
we denote by
\[
  \big\langle D_{ij}\mid 0\leq i \leq j\leq m-1\big\rangle
\]
the spanning subgraph of $K_{mn}$
containing exactly the edges
$\big\{ (g,i), (g+d,j) \big\}$ for every $g\in G$, $d\in D_{ij}$,
and $0\leq i\leq j\leq m-1$. The reader can easily check that this graph  remains unchanged if we replace any set $D_{ii}$ with $\pm D_{ii}$.

The following result, standard in the context of difference families,
provides us with a method to construct
$\G$-decompositions for subgraphs of $K_{mn}+w$.

\begin{prop}\label{mixed diff method}
Let $G$ be an abelian group of odd order $n$, let $m$ and $w$ be non-negative integers,
and denote by $\mathcal{F}$ a family of $k$-cycles
(resp., $k$-suns) with vertices in $(G\times [0, m-1]) \ \cup\ \{\infty_u\mid u\in\Z_w\}$
satisfying the following conditions:
  \begin{enumerate}
    \item $\Delta_{ij} \mathcal{F}$ has no repeated elements, for every $0\leq i \leq j < m$;
    \item $N_{\mathcal{F}}(\infty_u)=\big\{(g_{u,i},i)\mid 0\leq i <m, g_{u,i}\in G\big\}$
    for every  $1\leq{u}\leq{w}$.
  \end{enumerate}
Then $\bigcup_{\G\in\mathcal{F}} Orb_{G}(\G) = \{\tau_g(\G)\mid g\in G, \G\in\cF\}$ is a $k$-cycle (resp., $k$-sun) system of
$\langle \Delta_{ij}\cF \mid 0\leq{i}\leq{j}\leq m-1 \rangle+w$.
\end{prop}
\begin{proof} Let $\cF^* = \bigcup_{\G\in\mathcal{F}} Orb_{G}(\G)$, $K = \langle \Delta_{ij}\cF \mid 0\leq{i}\leq{j}\leq m-1 \rangle$, and let $\epsilon$ be an edge of $K + w$.
We are going to show that $\epsilon$ belongs to exactly one graph of $\cF^*$.

If $\epsilon\in E(K)$, by recalling the definition of $K$ we have that $\epsilon = \{(g,i), (g + d, j)\}$ for some $g\in G$ and $d\in\Delta_{ij} \cF$, with $0\leq i\leq j < m$.
Hence, there is a graph $\G\in\cF$ such that $d\in \Delta_{ij}\G$.
This means that $\G$ contains the edge $\epsilon' = \{(g',i), (g' + d, j)\}$ for some $g'\in G$,
therefore $\epsilon = \tau_{g-g'}(\epsilon') \in \tau_{g-g'}(\G) \in \cF^*$.
To prove that $\epsilon$ only belongs to $\tau_{g-g'}(\G)$, let $\G'$ be any graph in $\cF$ such that $\epsilon\in \tau_x(\G')$, for some $x\in G$.
Since translations preserve differences, we have that
$d\in \Delta_{ij}\tau_x(\G') = \Delta_{ij} \G'$.
Considering that $d\in \Delta_{ij} \G \ \cap\ \Delta_{ij} \G'$ and,
by condition (1),
$\Delta_{ij} \cF$ has no repeated elements, we necessarily have that $\G'=\G$, hence $\tau_{-x} (\epsilon) \in \G$. Again, since $\Delta_{ij} \Gamma$ has no repeated elements
(condition (1)),
and considering that $\epsilon'$ and $\tau_{-x}(\epsilon)$ are edges of $\Gamma$ that yield the same
differences, then $\tau_{-x}(\epsilon) = \epsilon' = \tau_{g'-g}(\epsilon)$, that is,
$\tau_{g'-g+x}(\epsilon) = \epsilon$. Since $G$ has odd order, it has no element of order $2$,
hence $g'-g+x=0$, that is, $x=g-g'$, therefore $\tau_{g-g'}(\Gamma)$ is the only graph of $\cF^*$ containing $\epsilon$.

Similarly, we show that every edge of $(K+w)\setminus{K}$ belongs to exactly one graph of $\cF^*$.
Let $\epsilon = \{\infty_u, (g,i)\}$ for some $u\in \Z_{w}$ and $(g,i)\in G \times [0, m-1]$.
By assumption, there is a graph $\Gamma\in\cF^*$ containing the edge
$\epsilon' = \{\infty_u, (g_{u,i},i)\}$ with $g_{u,i}\in G$.
Hence, $\epsilon = \tau_{g-g_{u,i}}(\epsilon')\in \tau_{g-g_{u,i}}(\Gamma)$.
Finally, if $\epsilon\in \tau_x(\Gamma')$ for some $x\in G$ and $\Gamma'\in\cF$, then
$\{\infty_u, (g-x,i)\}=\tau_{-x}(\epsilon)\in \Gamma'$.
Since condition (2) implies that
$N_{\mathcal{F}}(\infty_u)$ contains exactly one pair from $G\times \{i\}$,
we necessarily have that $\G=\G'$ and $x=g-g_{u,i}$; therefore,
there is exactly one graph of $\cF^*$ containing $\epsilon$.
Condition (2) also implies that $N_{\mathcal{F}}(\infty_u)$ is disjoint from
$\{\infty_u\mid u\in\mathbb{Z}_w\}$, and this guarantees that no graph in $\mathcal{F}^*$
contains edges joining two infinities.
Therefore, $\mathcal{F}^*$ is the desired decomposition of $K+w$.
\end{proof}

Considering that $K_{mn}=\langle D_{ij}\mid 0\leq i \leq j\leq m-1\rangle$ if and only if
$\pm D_{ii} = G\setminus\{0\}$ for every $i\in[0, m-1]$, and $D_{ij}=G$ for every $0\leq i<j\leq m-1$,
the proof of the following corollary to Proposition \ref{mixed diff method} is straightforward.

\begin{cor}\label{mixed diff method for Kmn}
  Let $G$ be an abelian group of odd order $n$, let $m$ and $w$ be non-negative integers,
  and denote by $\mathcal{F}$ a family of $k$-cycles
  (resp., $k$-suns) with vertices in $(G\times [0, m-1]) \ \cup\ \{\infty_u\mid u\in\Z_w\}$
  satisfying the following conditions:
  \begin{enumerate}
    \item $\Delta_{ij} \mathcal{F} =
    \begin{cases}
      G\setminus\{0\} & \text{if $0\leq i=j \leq m-1$};\\
      G &               \text{if $0\leq i<j\leq m-1$};
    \end{cases}
    $
    \item $N_{\mathcal{F}}(\infty_u)=\big\{(g_{u,i},i)\mid 0\leq i <m, g_{u,i}\in G\big\}$
    for every  $1\leq{u}\leq{w}$.
  \end{enumerate}
Then $\bigcup_{\G\in\mathcal{F}} Orb_{G}(\G)$ is a $k$-cycle (resp., $k$-sun) system of $K_{mn} + w$.
\end{cor}

\section{Constructing $k$-cycle systems of $\langle D_{00}, D_{01}, D_{11} \rangle + w$}
In this section, we recall and generalize some results from \cite{HLR} in order to provide  conditions on
$D_{00}, D_{01}, D_{11}\subseteq \Z_k$ that guarantee the existence of
a $k$-cycle system for the subgraph $\langle D_{00}, D_{01}, D_{11} \rangle + w$ of
$K_{2k}+w$, where $V(K_{2k}) = \Z_k \times \{0,1\}$.

We recall that every connected 4-regular Cayley graph over an abelian group has a Hamilton cycle system \cite{BeFaMa89}
and show the following.

\begin{lem}\label{ab, 0, 0} Let $[a,b], [c,d]\subseteq [1,\ell]$.
The graph $\big\langle  \left[a,b\right], \varnothing, \left[c,d\right] \big\rangle$
has a $k$-cycle system whenever both $[a,b]$ and $[c,d]$ satisfy the following condition:
the interval has even size or contains an integer coprime with $k$.
\end{lem}
\begin{proof}
The graph $\big\langle  \left[a,b\right], \varnothing, \left[c,d\right] \big\rangle$ decomposes into $\big\langle  \left[a,b\right], \varnothing,  \varnothing \big\rangle$
and $\big\langle   \varnothing, \varnothing, \left[c,d\right] \big\rangle$.
The first one is the Cayley graph
$\Gamma = \Cay(\Z_k, [a,b])$ with further $k$ isolated vertices, while the second one
is isomorphic to $\big\langle   \left[c,d\right], \varnothing, \varnothing \big\rangle$.
Therefore,  it is enough to show that $\Gamma$ has a $k$-cycle system.

Note that $\G$ decomposes into the subgraphs $\Cay(\Z_k, D_i)$, for $0\leq i\leq t$, whenever
the sets $D_i$ between them partition $\left[a, b\right]$. By assumption,
$[a,b]$ has even size or contains an integer coprime with $k$. Therefore, we can assume that
for every $i>0$ the set $D_i$ is a pair of integers at distance 1 or 2, and
$D_0$ is either empty or contains exactly one integer coprime with $k$.
Clearly, $\Cay(\Z_k, D_0)$ is either the empty graph or a $k$-cycle, and
the remaining $\Cay(\Z_k, D_i)$
are 4-regular Cayley graphs. Also, for every $i>0$ we have that
$D_i$ is a generating set of $\Z_k$
(since $k$ is odd and $D_i$ contains integers at distance 1 or 2), hence
the graph $\Cay(\Z_k, D_i)$ is connected. It follows that each $\Cay(\Z_k, D_i)$, with $i>0$, decomposes into two $k$-cycles, thus the assertion is proven.
\end{proof}

\begin{lem}\label{l, S S+1, 0}

Let $S \subseteq \{2i-1\mid 1\leq i\leq \ell\}$.
Then
there exist $k$-cycle systems for the graphs
$\big\langle  \{\ell\}, S\ \cup\ (S+1), \varnothing \big\rangle$
and $\big\langle  \{\ell\}, (S+1)\ \cup\ (S+2), \varnothing \big\rangle$.
\end{lem}
\begin{proof}
We note that the result is trivial when $S=\varnothing$, since
$\big\langle  \{\ell\}, \varnothing, \varnothing \big\rangle$ is a $k$-cycle.

The existence of a $k$-cycle system of
$\G = \big\langle  \{\ell\}, S\ \cup\ (S+1), \varnothing \big\rangle$ has been proven in
\cite[Lemma 3]{HLR} when $S\subseteq\{2i-1\mid 1\leq i\leq \ell\}$. Consider now the permutation
$f$ of $\Z_{k}\times\{0,1\}$ fixing $\Z_{k}\times\{0\}$ pointwise, and mapping $(i,1)$ to
$(i+1,1)$ for every $i\in \Z_k$. It is not difficult to check that $f(\G) = \big\langle  \{\ell\}, (S+1)\ \cup\ (S+2), \varnothing \big\rangle$
which is therefore isomorphic to $\G$, and hence it has a $k$-cycle system.
\end{proof}

\begin{lem}\label{+r}\label{0, D, 0}
Let $r,s$ and $s'$ be integers such that
$1\leq s\leq s' \leq \min\{s+1,\ell\}$,  and $0<r \not\equiv s+s' \pmod{2}$. Also, let
$D\subseteq [0,k-1]$ be a non-empty interval of size $k-(s+s'+2r)$. Then
 there is a cycle $C=(x_1, x_2, \ldots, x_k)$ of
 $\G = \big\langle [1+\epsilon,s+\epsilon], D, [1+\epsilon, s'+\epsilon] \big\rangle + r$, for every
 $\epsilon\in\{0,1\}$,
 such that $Orb(C)$ is a $k$-cycle system of $\G$.
  Furthermore, if $u=0$ or $u=1-\epsilon=1 \leq s-1$, then
   \begin{enumerate}
   \item $Dev\big(\{x_{2-u}, x_{3-u}\}\big)$ is a $k$-cycle with vertices in $\Z_k\times \{0\}$;
   \item $Dev\big(\{x_{4+u}, x_{5+u}\}\big)$ is a $k$-cycle with vertices in $\Z_k\times \{1\}$.
 \end{enumerate}
\end{lem}
\begin{proof}
Set $t = k-(s+s'+2r)$ and let
$\Omega = \big\langle [1+\epsilon,s+\epsilon], [0,t-1], [1+\epsilon, s'+\epsilon] \big\rangle + r$.
For $i\in [0, s+s'+1]$ and $j\in[0, t+r-1]$, let $a_i$ and $b_j$ be the elements of $\Z_{k}\times\{0,1\}$ defined
as follows:
 \begin{align*}
  a_i &=
  \begin{cases}
    \left(-\frac{i}{2}, 0\right)     & \text{if $i\in[0, s]$ is even},\\
    \left(-s -\epsilon +\frac{i-1}{2}, 0\right)
  & \text{if $i\in[1,s]$ is odd},\\
    a_{2s+1-i} + (0,1)               & \text{if $i\in [s+1, 2s+1]$},\\
    (-s'-\epsilon,1)
  & \text{if $i=s+s'+1 > 2s+1$},
  \end{cases} \\
  b_j &=
  \begin{cases}
    \left(\frac{j}{2}, 0\right)                    & \text{if $j\in[0, t +r -2]$ is even},\\
    \left(t-\frac{j+1}{2}, 1\right)            & \text{if $j\in[1, t-1]$ is odd},\\
    \left(t+ \left\lfloor\frac{j-t}{2}\right\rfloor, 1\right)   & \text{if $j\in [t, t+r-2]$ is odd},\\
    a_{s+s'+1}                                 & \text{if $j=t+r-1$}.
  \end{cases}
  \end{align*}
Since the elements $a_i$ and $b_j$ are pairwise distinct, except for $a_0=b_0$ and $a_{s+s'+1} = b_{t+r-1}$,
then the union $F$ of the following two paths is a $k$-cycle:
\begin{align*}
  P &= a_0 \sim a_1 \sim \ldots \sim a_{s+s'+1}, \\
  Q &= b_0 \sim b_1 \sim \ldots \sim b_{t-1}\sim \infty_1 \sim b_t \sim \infty_2 \sim b_{t+1} \sim \ldots \sim
  \infty_r \sim b_{t+r-1}.
\end{align*}
Since $\Delta_{ij} F =\Delta_{ij} P\ \cup\ \Delta_{ij} Q$, for $i,j\in\{0,1\}$, where
\[
\begin{array}{lll}
   \Delta_{00} P = \pm [1 +\epsilon, s+\epsilon],
 & \Delta_{01} P =  \{0\},
 & \Delta_{11} P = \pm [1 +\epsilon, s' +\epsilon],\\
   \Delta_{00} Q = \varnothing,
 & \Delta_{01} Q = [1,t-1],
 & \Delta_{11} Q = \varnothing,
\end{array}
\]
and considering that $N_F(\infty_h) = N_Q(\infty_h) = \{b_{t+h-2}, b_{t+h-1}\}$ for every
$h\in[1,r]$,
Proposition \ref{mixed diff method}
guarantees that $Orb(F)$ is a $k$-cycle system of $\Omega$. Furthermore,
if $u=0$ or $u=1-\epsilon=1\leq s-1$, then
\[
  \pm(a_{s-u} - a_{s-u-1}) = \pm (a_{s+u+2} - a_{s+u+1}) = \pm (u+\epsilon+1,0).
\]
Since $k$ is odd, we have that
$Dev( \{a_{s-u-1}, a_{s-u}\} )$ and $Dev( \{a_{s+u+2}, a_{s+u+1}\} )$ are $k$-cycles with
vertices in $\Z_{k}\times\{0\}$ and $\Z_{k}\times\{1\}$, respectively.

If $D=[g,g+t-1]$ is any interval of $[0, k-1]$ of size $t$, and $f$ is the permutation of $\Z_{k}\times\{0,1\}$
fixing $\Z_{k}\times\{0\}$ pointwise, and mapping $(i,1)$ to $(i + g, 1)$ for every $i\in \Z_k$,  one can check that $C=f(F)$ is the desired $k$-cycle of $\G=f(\Omega)$.
\end{proof}

\begin{lem}\label{+l} \

  \begin{enumerate}
    \item Let $\ell$ be odd.
    If $\G$ is a $1$-factor of $K_{2k}$, then $\G+\ell$ decomposes into $k$ cycles of length $k$,
      each~of~which contains exactly one edge of $\G$.
      Furthermore, if $\G = \big\langle \varnothing, \{d\}, \varnothing \big\rangle$, then there exists a
      $k$-cycle $C=(c_1, c_2, \ldots, c_{k})$ of $\G+\ell$, with $c_1 \in \Z_{k}\times\{0\}$ and
      $c_2 \in \Z_{k}\times\{1\}$, such that
      \[
      \text{$Dev(\{c_1, c_2\})=\G$ and $Orb(C)$ is a $k$-cycle system of $\G+\ell$}.
      \]
    \item Let $\ell$ be even.
      If $\G$ is a $k$-cycle of order $2k$,
      then $\G+\ell$ decomposes into $k$ cycles of length $k$, each of which contains exactly one edge of $\G$.
      Furthermore, if $\G = \big\langle \{d\}, \varnothing, \varnothing \big\rangle$ and $d$ is coprime with $k$,
      then there exists a $k$-cycle $C=(c_1, c_2, \ldots, c_{k})$ of $\G+\ell$,
      with $c_1,c_2 \in \Z_{k}\times\{0\}$, such that
      \[
      \text{$Dev(\{c_1, c_2\})$ is the $k$-cycle of $\G$ and $Orb(C)$ is a $k$-cycle system of $\G+\ell$}.
      \]
  \end{enumerate}
\end{lem}
\begin{proof} Permuting the vertices of $K_{2k}$ if necessary, we can assume that
$\G$ is the $1$-factor $\G_0= \big\langle \varnothing, \{0\}, \varnothing \big\rangle$ when $\ell$ is odd,
and the $k$-cycle $\G_1=\big\langle \{\ell\},\varnothing, \varnothing \big\rangle$ (of order $2k$) when $\ell$ is even.
For $h\in\{0, 1\}$,
let $C_h = (c_{h,1}, c_{h,2}, \infty_1, c_3, \infty_2, c_4, \ldots, \infty_{\ell-1}, c_{\ell+1}, \infty_\ell)$ be
the $k$-cycle of $\G_h + \ell$, where
\[
  c_{h,1} = \big(0, 1-h\big),\; c_{h,2} = \big(h\ell, 0\big),\;\text{and}\;
  c_j =
  \begin{cases}
    \left( \frac{j-1}{2}, 1\right) & \text{if $j\in[3, \ell+1]$ is odd}, \\
    \left( \frac{j}{2}, 0\right)   & \text{if $j\in[4, \ell+1]$ is even}.
  \end{cases}
\]
Note that the sets $\Delta_{ij} C_h$ are empty, except for $\Delta_{01} C_{0}=\{0\}$
and $\Delta _{00} C_1 = \{\pm\ell\}$. Also, the two neighbours of
$\infty_u$ in $C_h$ belong to $\Z_{k}\times \{0\}$ and $\Z_{k}\times \{1\}$, respectively.
Hence, Proposition \ref{mixed diff method} guarantees that
$Orb(C_h)$ is a $k$-cycle system of $\G_h+\ell$, for $h\in\{0,1\}$. We finally notice that
$Dev(\{c_{h,1}, c_{h,2}\})= \G_h$ (up to isolated vertices)
and this completes the proof.
\end{proof}

The following result  has been proven in \cite{HLR}.

\begin{lem}\label{HLR:1factorization}
  Let $D\subseteq \left[1, \ell\right]$.
  The subgraph $\langle D,\{0\}, D \rangle$ of $K_{2k}$ has a $1$-factorization.
\end{lem}

\begin{rem}\label{rem1}
  Considering the permutation $f$ of $\Z_{k}\times \{0,1\}$ such that $f(i, j) = (i, 1-j)$, and
  a graph $\G=\big\langle D_0, D_1, D_2\big\rangle$, we have that $f(\G)=\big\langle D_2, -D_1, D_0\big\rangle$. Therefore,
  Lemmas \ref{ab, 0, 0} --  \ref{HLR:1factorization} 
  continue to hold when we replace $\G$ by $f(\G)$.
\end{rem}

\section{$k$-sun systems of $K_{4k}+n$}

In this section we provide sufficient conditions for a $k$-sun system of
$K_{4k}+n$ to exist, when $n\equiv 0,1 \pmod{4}$. More precisely, we show the following.
\begin{thm}\label{thm:hole}
  Let $k\geq7$ be an odd integer and let $n\equiv 0,1\pmod4$ with $2k< n<10k$, then  there exists a
  $k$-sun system of $K_{4k}+n$, except possibly when
  \begin{itemize}
    \item $k=7$ and $n=20,21,32,33,44,45,56,57,64,65,68,69$,
    \item $k=11$ and $n=100,101,112,113$.
  \end{itemize}
\end{thm}
To prove Theorem \ref{thm:hole}, we start by introducing some notions and prove some preliminary results.
Let $M$ be a positive integer and
take $V(K_{2^iM})=\Z_M\times [0,2^i-1]$ and
$V(K_{2^{i}M} + w) = V(K_{2^iM})\ \cup\ \big\{\infty_h\mid h\in\Z_w\big\}$,
for $i\in\{1,2\}$ and $w>0$.

Now assume that $w=2u$, and let $x\mapsto \ov{x}$ be the permutation of $V(K_{4M} + 2u)$ defined as follows:
\[
\ov{x} =
\begin{cases}
  (a, 2-j) & \text{if $x = (a,j)\in \Z_{M}\times \{0,2\}$,}\\
  (a, 4-j) & \text{if $x = (a,j)\in \Z_{M}\times \{1,3\}$,}\\
  \infty_{h+u} &   \text{if $x=\infty_h$}.
\end{cases}
\]
For any subgraph $\G$ of $K_{4M}+2u$,
we denote by $\overline{\G}$ the graph (isomorphic to $\G$)
obtained by replacing each vertex $x$ of $\G$ with $\overline{x}$.

Given a subgraph $\G$ of $K_{2M}+u$, we denote by $\G[2]$ the spanning subgraph of $K_{4M}+2u$
whose edge set is
\[
E(\G[2]) = \big\{\{x, y\}, \{x, \ov{y}\}, \{\ov{x}, y\}, \{\ov{x}, \ov{y}\} \mid \{x,y\}\in E(\G) \big\},
\]
and let $\G^*[2] = \G[2] \oplus I$  be the graph obtained by adding to $\G[2]$
the \emph{$1$-factor}
\[I=\big\{ \{x, \ov{x}\} \mid x\in\Z_{M}\times\{0,1\}\big\}.
\]
Note that, up to isolated vertices, $\G[2]$ is the \emph{lexicographic product} of
$\Gamma$ with the empty graph on two vertices.

The proof of the following elementary lemma is left to the reader.

\begin{lem}\label{trivial}
Let $\G=\oplus_{i=1}^n \G_{i}$ and let $w=\sum_{i=1}^n w_i$ with $w_i\geq 0$.
If $\G$ and the $\G_i$s have the same vertex set (possibly with isolated vertices), then
\begin{enumerate}
 \item  $\G + w  = \oplus_{i=1}^n (\G_i + w_i)$;
 \item  $\G[2] = \oplus_{i=1}^n \G_i[2]$;
 \item  $(\G+w)[2] = \G[2] + 2w$.
\end{enumerate}
\end{lem}

We start showing that if $C$ is a $k$-cycle, then $C[2]$ decomposes into two $k$-suns.

\begin{lem}\label{C[2]}
Let $C=(c_1, c_2, \ldots, c_k)$ be a cycle with vertices in
$\big(\Z_M \times\{0,1\}\big) \ \cup \ \{\infty_h \mid h\in \Z_{u}\}$ and let $S$ be the $k$-sun defined as follows:
\begin{equation}\label{fromCtoS}
  S=\left(
  \begin{array}{cccc}
    s_1      & \ldots & s_{k-1}      & s_k \\
    \ov{s_2} & \ldots & \ov{s_{k}}   & \ov{s_1}
  \end{array}
  \right)
\end{equation}
where $s_i \in \{c_i, \overline{c_i}\}$ for every $i\in[1,k]$.
Then $C[2] = S \oplus \overline{S}$.
\end{lem}
\begin{proof}
It is enough to notice that $S$ contains the edges
$\{s_{i}, s_{i+1}\}$ and $\{s_{i}, \ov{s_{i+1}}\}$,
while $\ov{S}$ contains $\{\ov{s_{i}}, \ov{s_{i+1}}\}$ and $\{\ov{s_{i}}, s_{i+1}\}$, for every $i\in[1,k]$, where $s_{k+1}=s_1$ and $\ov{s_{k+1}} = \ov{s_1}$.
\end{proof}

\begin{ex}
  In Figure \ref{figura} we have the graph $C_7[2]$ which can be decomposed into two $7$-suns $S$ and $\ov{S}$.
The non-dashed edges are those of $S$, while the dashed edges are those of $\ov{S}$.
  \begin{figure}[ht]
    \centering
    \includegraphics[scale=0.17]{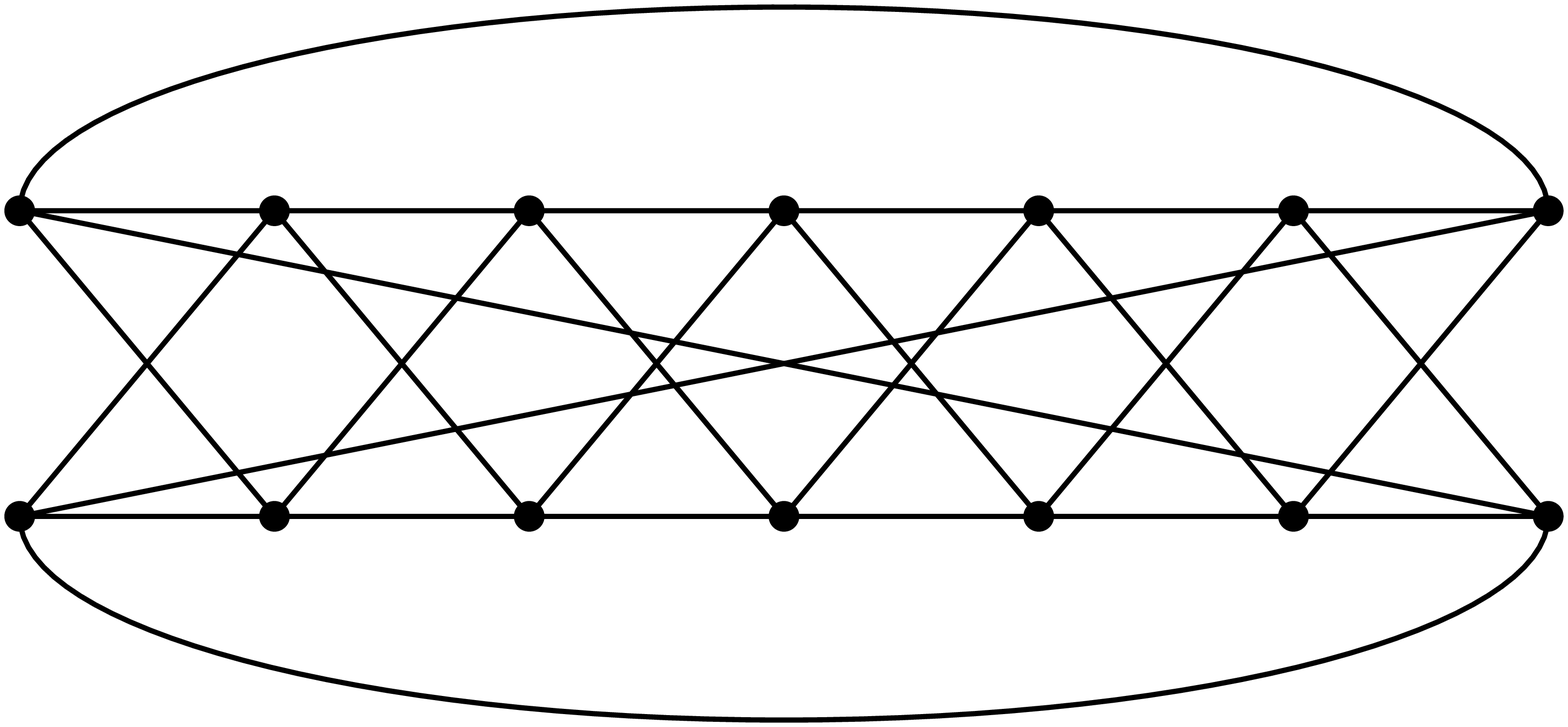}
    \hspace{1cm}
       \includegraphics[scale=0.17]{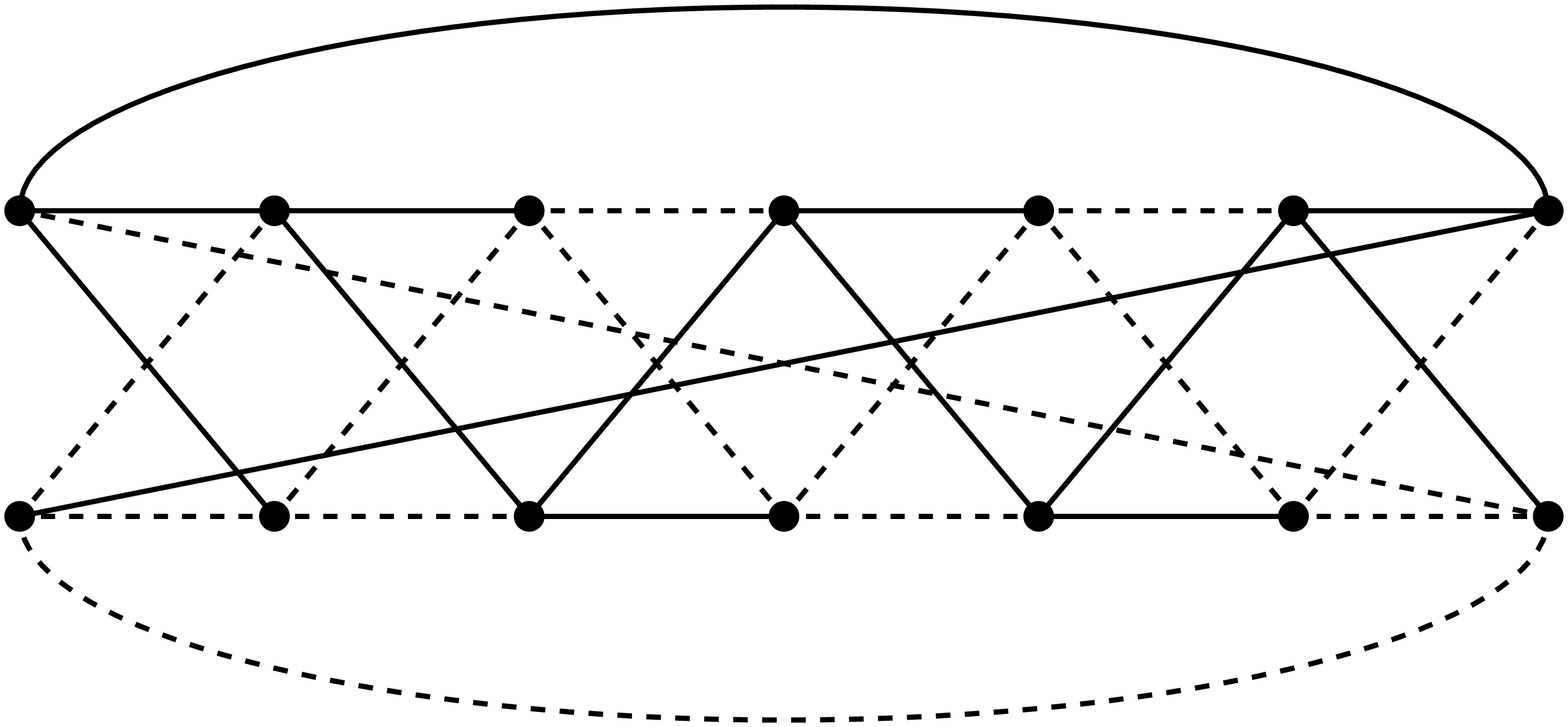}
    \caption{$C_7[2]=S\oplus \ov{S}$}
    \label{figura}
  \end{figure}
\end{ex}

For every cycle $C= (c_1, c_2, \ldots, c_{k})$ with vertices in $\Z_M \times\{0,1\}$, we set
\[
  \sigma(C)=
  \left(
  \begin{array}{cccc}
    c_1      & \ldots & c_{k-1}      & c_k \\
    \ov{c_2} & \ldots & \ov{c_{k}}   & \ov{c_1}
  \end{array}
  \right).
\]
Clearly, $C[2] = \sigma(C) \oplus \ov{\sigma(C)}$ by Lemma \ref{C[2]}.

\begin{lem}\label{from k-C to k-S}
If $\cC=\{C_1, C_2, \ldots, C_t\}$ is a $k$-cycle system of $\G+u$, where $\G$ is a subgraph of
$K_{2M}$, and
$S_i$ is a $k$-sun obtained from $C_i$ as in Lemma \ref{C[2]}, then
$\cS = \big\{S_i, \ov{S_i}\mid i\in[1,t]\big\}$ is a $k$-sun~system of $\G[2]+2u$. In particular,
if $\cC = Orb\big(C_1\big)$, then
$Orb(S_1) \ \cup\ Orb\big(\ov{S_1}\big)$
is a $k$-sun system of $\G[2]+2u$.
\end{lem}
\begin{proof}
By assumption $\Gamma + u = \oplus_{i=1}^t C_i$, where each $C_i$ is a $k$-cycle. Also, by Lemma \ref{trivial},
we have that $\G[2] + 2u = (\G + u)[2] = \oplus_{i=1}^t C_i[2]$. Since $C_i[2] = S_i \oplus \ov{S_i}$ by Lemma
\ref{C[2]}, then $\cS$ is a $k$-sun system of $\G[2]+2u$.

The second part easily follows by noticing that
if $C_i = \tau_g(C_1)$ for some $g\in\Z_M$, then
$C_i[2] = \tau_g(C_1[2]) = \tau_g\big(S_1\big) \oplus \tau_g\big(\ov{S_1}\big)$.
\end{proof}

The following lemma describes the general method we use to construct $k$-sun systems of $K_{4k}+n$.
We point out that throughout the rest of this section we take
$V(K_{2k}) = \Z_{k}\times \{0,1\}$ and $V(K_{4k}) = \Z_{k}\times [0,3]$.

\begin{lem}\label{lemma:main}
Let $K_{2k}= \G_1\oplus \G_2$ with $V(\G_1) = V(\G_2) = V(K_{2k})$.
If $\G_1 + w_1$ has a $k$-cycle system and $\G_2^*[2] + w_2$ has a $k$-sun system,
then $K_{4k} + (2w_1+w_2)$ has  a $k$-sun system.
\end{lem}
\begin{proof} The result follows by Lemma \ref{trivial}. In fact, noting that
$K_{4k} = K_{2k}[2] \oplus I$, where $I=\big\{ \{z, \ov{z}\} \mid z\in\Z_{k}\times \{0,1\}  \big\}$, we have that
\begin{align*}
  K_{4k} + (2w_1 + w_2) &= \big(\G_1[2] \oplus (\G_2[2] \oplus I)\big) + 2w_1 + w_2  \\
                              &= (\G_1[2] + 2w_1) \oplus (\G_2^*[2] + w_2) = (\G_1+ w_1)[2] \oplus (\G_2^*[2] + w_2).
\end{align*}
The result then follows by Lemma \ref{from k-C to k-S}.
\end{proof}

We are now ready to prove the main result of this section, Theorem \ref{thm:hole}.
The case $k\equiv 1 \pmod{4}$ is proven in Theorem \ref{thm:holek=1}, while the
case $k\equiv 3 \pmod{4}$ is dealt with in Theorems
\ref{thm:holek=3}, \ref{thm:holek=3,a}, \ref{thm:holek=3, b1} and \ref{thm:holek=3,b2}.

\begin{thm}\label{thm:holek=1}
If $k\equiv 1\pmod 4\geq 9$ and $n\equiv 0,1 \pmod 4$ with  $2k < n< 10k$, then there exists a
$k$-sun system of $K_{4k}+n$.
\end{thm}
\begin{proof} Let $n = 2(q\ell + r) + \nu$ with $1\leq r\leq \ell$ and $\nu\in\{2,3\}$.
Note that  $\ell\geq 4$ is even and $r$ is odd, since $n\equiv 0,1 \pmod 4\geq9$ and $k\equiv 1\pmod 4$.
Considering also that $2k < n<10k$, we have that $2\leq q\leq 10\leq k+2r-1$. Furthermore,
let $V(K_{4k}+n) = \big(\Z_{k}\times[0,3]\big)\cup\{\infty_h\mid h\in \Z_{n-\nu}\}
\cup \{\infty'_1,\infty'_2, \infty'_\nu\}$.

We start decomposing $K_{2k}$ into the following two graphs:
\begin{align*}
  \G_1 = \big\langle [2,\ell], [k-2r-2, k-1], [2,\ell-1] \big\rangle \;\;\mbox{and}\;\;
  \G_2 = \big\langle \{1\}, [0,k-2r-3], \{1,\ell\} \big\rangle.
\end{align*}
We notice that $\G_1$ further decomposes into the following graphs:
\[
  \big\langle [2,\ell-1], \varnothing, \varnothing \big\rangle,  \hspace{.3cm}
  \big\langle \varnothing, \varnothing, [2,\ell-1] \big\rangle,   \hspace{.3cm}
  \big\langle  \{\ell\}, [k-2r-2, k-1], \varnothing \big\rangle,
\]
each of which decomposes into $k$-cycles by Lemmas \ref{ab, 0, 0} and \ref{l, S S+1, 0}; hence $\G_1$ has a $k$-cycle system
$\{C_1, C_2, \ldots, C_{\gamma}\}$, where $\gamma = k+2r-2$.
Note that this system is non-empty, since $1\leq q-1 \leq \gamma$.
Without loss of generality, we can assume that each cycle $C_i$ has order $2k$ and
\begin{align}
  & \label{thm:holek=1:cond1}
  \text{$C_1$ is a subgraph of $\big\langle [2,\ell-1], \varnothing, \varnothing \big\rangle$.}
\end{align}
Now set $\Omega_1 = \G_1\setminus C_1$ and $\Omega_2 = \G_2 \oplus C_1$.
Letting $w_1 = (q-2)\ell = \sum_{j=2}^{\gamma} w_{1,j}$, where
$w_{1,j} = \ell$ when $j <q$, and $w_{1,j} = 0$ otherwise,
by Lemma \ref{trivial} we have that $\Omega_1 + w_1 = \oplus_{i=2}^\gamma (C_i + w_{1,i})$.
Therefore, $\Omega_1 + w_1$ has a $k$-cycle system, since each $C_i + w_{1,i}$ decomposes into $k$-cycles
by Lemma \ref{+l}.
Setting $w_2 = n - 2w_1 = 2(2\ell+r) + \nu$ and considering that $K_{2k} = \G_1 \oplus \G_2 = \Omega_1 \oplus \Omega_2$,
by Lemma \ref{lemma:main} it is left to show that
$\Omega_2^*[2] + w_2$ has a $k$-sun system.

Set $\G_3=C_1$, and recall that $\Omega_2^*[2]= \Omega_2[2] \oplus I = \G_2[2] \oplus \G_3[2] \oplus I$,
where $I$ denotes the $1$-factor $\big\{ \{z, \ov{z}\} \mid z\in \Z_k \times\{0,1\}\big\}$ of $K_{4k}$.
Hence,
\begin{equation}\label{omega2}
\Omega_2^*[2] + w_2 = \big(\G_2 +(\ell+r)\big)[2] \oplus (\G_3 + \ell)[2] \oplus (I+\nu)
\end{equation}
by Lemma \ref{trivial}.
Clearly, $\G_2 = \G_{2,1} \oplus \G_{2,2}$ where
$\G_{2,1} = \big\langle \{1\}, [0, k-2r-3], \{1\} \big\rangle$ and
$\G_{2,2} = \big\langle \varnothing, \varnothing, \{\ell\}\big\rangle$,
hence $\G_2 +(\ell+r) = (\G_{2,1} + r) \oplus (\G_{2,2} + \ell)$.
By Lemmas \ref{+r} and \ref{+l}, there exist a $k$-cycle $A = (x_1, x_2, y_3, y_4, a_5, \ldots, a_k)$ of $\G_{2,1} + r$ and a $k$-cycle $B = (y_1, y_2, b_3, \ldots, b_k)$ of
$\G_{2,2} + \ell$ satisfying the following properties:
\begin{align}
    & \label{thm:holek=1:cond2}
    \text{$Orb(A)\ \cup\ Orb(B)$ is a $k$-cycle system of $\G_2 +(\ell+r)$;} \\
    & \label{thm:holek=1:cond3}
    \text{$Dev(\{x_1, x_2\})$ is a $k$-cycle with vertices in $\Z_{k}\times\{0\}$;}\\
    & \label{thm:holek=1:cond4}
    \text{$Dev(\{y_1, y_2\})$ and $Dev(\{y_3, y_4\})$ are $k$-cycles with vertices in $\Z_{k}\times\{1\}$.}
\end{align}
Furthermore, denoted by $(c_1, c_2, \ldots, c_k)$ the cycle in $\G_3$, Lemma \ref{+l} guarantees that
\begin{align*}
  \begin{split}
   & \text{$\G_3 + \ell$ has a $k$-cycle system $\{F_1, F_2, \ldots, F_{k}\}$ such that} \\
   & \text{$F_{j} = (c_j, c_{j+1}, f_{j,3}, f_{j,4}, \ldots, f_{j,k})$ for every $j\in[1,k]$ (with $c_{k+1}=c_1$).}
  \end{split}
\end{align*}
Let $\cS = \{S_1, S_2, S_3, S_4\}$
and $\cS' = \{ S_{3+2j}, S_{4+2j} \mid j\in[1,k]\}$, where
\begin{align*}
  & S_1 = \sigma(x_1, \ov{x_2}, y_3, y_4, a_5, \ldots, a_k),\;\;
    S_3 = \sigma(y_1, \ov{y_2}, b_3, \ldots, b_k),\\
  & S_{3+2j} = \sigma(c_j,\ov{c_{j+1}}, f_{j,3}, f_{j,4}, \ldots, f_{j,k}) \;\; \text{for $j\in[1,k]$},
    \;\; \text{and}\;\;\\
  & S_{2i} = \ov{S_{2i-1}}\;\; \text{for $i\in[1,k+2]$}.
\end{align*}
By Lemma \ref{from k-C to k-S} we have that
$\bigcup_{S\in \cS} Orb(S)$ is a $k$-sun system of
$\big(\G_2 +(\ell+r)\big)[2]$, and
$\cS'$ is a $k$-sun system of
$\big(\G_3 + \ell\big)[2]$.
It follows by \eqref{omega2} that $\bigcup_{S\in \cS} Orb(S)\ \cup\ \cS'$ decomposes $(\Omega_2^*[2] + w_2)\setminus (I+\nu)$.

To construct a $k$-sun system of $\Omega_2^*[2] + w_2$, we first modify the $k$-suns in $\cS\ \cup\ \cS'$
by replacing some of their vertices with $\infty'_1, \infty'_2$, and possibly $\infty'_3$ when $\nu=3$.
More precisely, following Table \ref{Table}, we obtain $T_i$ from $S_i$ by replacing the ordered set $V_i$ of vertices
of $S_i$ with $V_i'$. This yields a set $M_i$ of `missing' edges no longer covered by $T_i$ after this substitution,
but replaced by those in $N_i$, namely
\[
E(T_i) = \big(E(S_i)\setminus M_i\big) \ \cup\ N_i.
\]
We point out that  $T_{3+2j}=S_{3+2j}$,
and $T_{4+2j} = S_{4+2j}$ when $\nu=2$, for every $j\in[1,k]$. The remaining graphs $T_i$ are explicitly
given below, where the elements in bold are the replaced vertices.
\begingroup
\allowdisplaybreaks
\begin{align*}
T_1 &=
  \left(
  \begin{matrix}
    x_1            & \ov{x_2} & \bm{\infty'_2}   &     y_4   &     a_5   & \ldots & a_{k-1}    & a_k\\
    \bm{\infty'_1} & \ov{y_3} & \ov{y_4}         &  \ov{a_5} & \ov{a_6}  & \ldots & \ov{a_{k}} & \ov{x_{1}}
  \end{matrix}
  \right),\\
T_2 &=
\begin{cases}
  \left(
  \begin{matrix}
    \ov{x_1}       &  {x_2}         & \ov{y_3} &  \ov{y_4} & \ov{a_5} & \ldots & \ov{a_{k-1}} & \ov{a_{k}}\\
    \bm{\infty'_1} & \bm{\infty'_2} &  {y_4}   &   {a_5}   &     a_6  & \ldots & a_{k}        & x_1
  \end{matrix}
  \right) & \mbox{if $\nu=2$}, \\
  \left(
  \begin{matrix}
    \ov{x_1}  &  {x_2} & \bm{\infty'_3} &  \ov{y_4}  & \ov{a_5} & \ldots & \ov{a_{k-1}} & \ov{a_{k}}\\
    \bm{\infty'_1}    & \bm{\infty'_2} &  {y_4} &     {a_5}  &     a_6  & \ldots &     a_{k}    & x_1
  \end{matrix}
  \right) & \mbox{if $\nu=3$},
\end{cases}  \\
T_3 &=
  \left(
  \begin{matrix}
    y_1       & \ov{y_2} & b_3      & \ldots & b_{k-1}      &  b_k\\
    \bm{\infty'_1} & \ov{b_3} & \ov{b_4} & \ldots & \ov{b_{k}}   &  \ov{y_1}
  \end{matrix}
  \right),\;\;
T_4 =
  \left(
  \begin{matrix}
    \ov{y_1}    & y_2        & \ov{b_3} & \ldots & \ov{b_{k-1}} &  \ov{b_k}\\
    \bm{\infty'_1}   & b_3   & b_4      & \ldots & b_{k}        &  y_1
  \end{matrix}
  \right), \\
T_{4+2j} &=
  \left(
  \begin{matrix}
    \ov{c_j}     & c_{j+1} & \ov{f_{j,3}} & \ldots  & \ov{f_{j,k-1}} &  \ov{f_{j,k}} \\
    \bm{\infty'_3}    & f_{j,3} &     f_{j,4}  & \ldots  &       f_{j,k}  &  c_j
  \end{matrix}
  \right)\;\; \mbox{for every $j\in[1,k]$}.
\end{align*}
\endgroup
We notice that $\displaystyle
   \bigcup_{i=1}^{4} Dev(N_i) \cup \bigcup_{i=5}^{2k+4} N_i=
   \big\{ \{\infty'_j, x\} \mid j\in[1,\nu], x\in \Z_k\times [0,3] \big\}
$.
We finally build the following $2\nu +1$ graphs:
\[
\begin{array}{ll}
  G_1 =
  \begin{cases}
    Dev\big(x_1 \sim x_2 \sim \ov{x_2} \big) & \text{if $\nu=2$,} \\
    Dev\big(x_1 \sim x_2 \sim \ov{y_3} \big) & \text{if $\nu=3$,}
  \end{cases}
&
  G_2 = Dev\big( \ov{x_1} \sim \ov{x_2} \sim y_3 \big),
 \\[1ex] \rule{0pt}{1\normalbaselineskip}
  G_3 = Dev\big( y_4 \sim y_3 \sim x_2\big),
&
  G_4 = Dev\big({y_1} \sim {y_2} \sim \ov{y_2} \big),
\\[1ex] \rule{0pt}{1\normalbaselineskip}
  G_5 = Dev\big(\{\ov{y_1}, \ov{y_2}\} \oplus \{y_3, \ov{y_4}\} \big),
&
  G_6 = Dev\big( \ov{y_4} \sim \ov{y_3} \sim y_4 \big),
\\[1ex] \rule{0pt}{1\normalbaselineskip}
  G_7 =
  \left(
  \begin{matrix}
   \ov{c_1}     & \ov{c_{2}} & \ldots  & \ov{c_{k}}    \\
       c_1      &      c_{2} & \ldots  &      c_{k}
  \end{matrix}
  \right).
\end{array}
\]
By recalling \eqref{thm:holek=1:cond1} and
\eqref{thm:holek=1:cond2}--\eqref{thm:holek=1:cond4}, it is not difficult to check that
$G_1, G_2, \ldots, G_{2\nu+1}$ are $k$-suns. Furthermore,
\[
\bigcup_{i=1}^{2\nu+1}E(G_i) = \bigcup_{i=1}^{4} Dev(M_i)\ \cup\ \bigcup_{i=5}^{2k+4} M_i\ \cup\ E(I),
\]
where, we recall, $I$ denotes the $1$-factor $\big\{ \{z, \ov{z}\} \mid z\in \Z_k \times\{0,1\}\big\}$ of $K_{4k}$.
Therefore, $\bigcup_{i=1}^4 Orb(T_i)\ \cup\ \{T_5, T_6,\ldots, T_{2k+4}\} \ \cup\ \{G_1, G_2, \ldots, G_{2\nu+1}\}$
is a $k$-sun system of $\Omega^*_2[2] + w_2$, and this concludes the proof.
\end{proof}
\begin{table}[h]
  \centering
  {\renewcommand\arraystretch{1.4}
  \begin{tabular}{|p{0.8cm}|p{4.3cm}|p{2.5cm}|p{2.8cm}|p{0.4cm}|}
  \hline
  $\bm{i}$ & $\bm{V_i\rightarrow V'_i}$ & $\bm{M_i}$ & $\bm{N_i}$ & $\bm{\nu}$\\
      \hline
  $1$ & $(x_2,y_3)\rightarrow (\infty'_1,\infty'_2)$
      & $\{x_1,x_2\}, \{\ov{x_2},y_3\}$ $\{y_3,y_4\},\{y_3,\ov{y_4}\}$
      & $\{\infty'_1,x_1\}, \{\infty'_2,\ov{x_2}\}$ $\{\infty'_2,y_4\}, \{\infty'_2,\ov{y_4}\}$
      & {2,3} \\
      \hline
  $2$ & $(\ov{x_2},y_3)\rightarrow (\infty'_1,\infty'_2)$
      & $\{\ov{x_1},\ov{x_2}\},\{x_2,y_3\}$
      & $\{\infty'_1,\ov{x_1}\},\{\infty'_2,x_2\}$
      & 2 \\
    \hline
  $2$ & $(\ov{x_2},y_3,\ov{y_3})\rightarrow(\infty'_1,\infty'_2,\infty'_3)$
      & $\{\ov{x_1},\ov{x_2}\}, \{x_2,y_3\}$ $\{x_2,\ov{y_3}\}, \{\ov{y_3},\ov{y_4}\}$ $\{\ov{y_3},y_4\}$
      & $\{\infty'_1,\ov{x_1}\}, \{\infty'_2,x_2\}$  $\{\infty'_3,x_2\}, \{\infty'_3,\ov{y_4}\}$ $\{\infty'_3, y_4\}$
      & 3 \\
 \hline
  $3$ & $y_2\rightarrow \infty'_1$
      & $\{y_1,y_2\}$
      & $\{\infty'_1,y_1\}$
      & {2,3} \\
    \hline
  $4$ & $\ov{y_2}\rightarrow \infty'_1$
      & $\{\ov{y_1},\ov{y_2}\}$
      & $\{\infty'_1,\ov{y_1}\}$
      & {2,3} \\
    \hline
  $3+2j$ & $\varnothing$
         & $\varnothing$
         & $\varnothing$
         & 2,3 \\
    \hline
  $4+2j$ & $\varnothing$
         & $\varnothing$
         & $\varnothing$
         & {2} \\
    \hline
  $4+2j$ & $\ov{c_{j+1}}\rightarrow \infty'_3$
         & $\{\ov{c_j},\ov{c_{j+1}}\}$
         & $\{\infty'_3,\ov{c_j}\}$
         & {3} \\
    \hline
  \end{tabular}
  \vspace{0.4cm}
  }
  \caption{From $S_i$ to $T_i$.}\label{Table}
\end{table}

\begin{ex} By following the proof of Theorem \ref{thm:holek=1}, we construct a
$k$-sun system of $K_{4k} + n$ when $(k,n) = (9,21)$; hence
$(\ell, q, r, \nu) = (4, 2, 1, 3)$.

The graphs $\G_1 = \big\langle [2,4], [5, 8], [2,3] \big\rangle$ and
$\G_2 = \big\langle \{1\}, [0,4], \{1,4\} \big\rangle$ decompose the complete graph $K_{18}$
with vertex-set $\Z_{9} \times \{0,1\}$. Also $\Gamma_1$ decomposes into the following
$9$-cycles of order $18$, where $i=0,1$:
\begin{align*}
  C_{1+i} &= ((0,i), (2,i), (8,i), (1,i), (3,i), (5,i), (7,i), (4,i), (6,i)),\\
  C_{3+i} &= ((0,i), (3,i), (6,i), (8,i), (5,i), (2,i), (4,i), (1,i), (7,i)),\\
  C_{5+i} &= ((4i,0), (8+4i,1), (1+4i,0), (4i,1), (2+4i,0), (1+4i,1),  \\
          &\;\;\;\;\;\; (3+4i,0), (2+4i,1), (4+4i,0)), \\
  C_{7+i} &= ((8+4i,0), (5+4i,1), (4i,0), (6+4i,1), (1+4i,0), (7 + 4i,1),  \\
          &\;\;\;\;\;\; (2+4i,0), (8+4i,1), (3+4i,0)), \\
      C_9 &= ((7,0), (2,0), (6,0), (1,0), (5,0), (0,0), (7,1), (8,0), (4,1)).
\end{align*}
Clearly, $K_{18}= \Omega_1 \oplus \Omega_2$, where $\Omega_1 = \Gamma_1\setminus{C_1}$ and $\Omega_2=\Gamma_2\oplus C_1$.

Let $V(K_{36})=\Z_9\times [0,3]$, and
denote by $I$ the $1$-factor of $K_{36}$ containing all edges of the form $\{(a,i),(a,i+2)\}$,
with $a\in\Z_9$ and $i\in\{0,1\}$. Then,
\[
  K_{36} = K_{18}[2] \oplus I = \Omega_1[2] \oplus \Omega_2[2] \oplus I.
\]
Considering that $(\Omega_2 + 9)[2] = \Omega_2[2] + 18$, we have
\[
  K_{36} + 21 =  \Omega_1[2] \oplus (\Omega_2[2] + 18) \oplus (I + 3)
              =  \Omega_1[2] \oplus (\Omega_2 + 9)[2] \oplus (I + 3)   .
\]
Since the set $\{\sigma(C_i), \ov{\sigma(C_i)} \mid i\in[2,9]\}$ is a $9$-sun system of
$\Omega_1[2]$, it is left to build a $9$-sun system of $\Omega^*_2[2] + 21 = (\Omega_2[2] + 18) \oplus (I + 3)$.

We start by decomposing $\Omega_2 + 9$ into $9$-cycles. Since
$\Omega_2 = \G_{2,1} \oplus \G_{2,2} \oplus \G_3$ with
$\G_{2,1} = \big\langle \{1\}, [0,4], \{1\} \big\rangle$,
$\G_{2,2} = \big\langle \varnothing, \varnothing, \{4\} \big\rangle$
and $\G_3 = C_1$, then
\[
  \Omega_2 + 9 =  (\G_{2,1}+1) \oplus (\G_{2,2}+4) \oplus (\G_3+4).
\]
Let $A=(x_1, x_2, y_3, y_4, a_5, \ldots, a_9)$ and $B=(y_1, y_2, b_3, \ldots, b_9)$
be the $9$-cycles defined as follows:
\begin{align*}
  (x_1, x_2, y_3, y_4) &= ((0,0), (-1,0), (-1,1), (0,1)), \\
    (a_5, \ldots, a_9) &= (\infty_1, (2,0), (3,1), (1,0), (4,1)),\\
  (y_1, y_2)           &= ((0,1), (4,1)), \\
  (b_3, \ldots, b_9)   &= (\infty_2, (1,0), \infty_3, (1,1), \infty_4, (0,0), \infty_5).
\end{align*}
One can easily check that $Orb(A)$ (resp., $Orb(B)$) decomposes $\G_{2,1}+1$
(resp., $\G_{2,2}+4$).
Also, for every edge $\{c_j, c_{j+1}\}$ of $C_1$,
with $j\in[1,9]$ and $c_{10}= c_1$, we construct the cycle $F_j=(c_j, c_{j+1}, f_{j,3}, f_{j,4}, \ldots, f_{j,9})$, where
\[(f_{j,3}, f_{j,4}, \ldots, f_{j,9})= (\infty_6, (1,0), \infty_7, (1,1), \infty_8, (0,0), \infty_9).
\]
One can check that $\{F_1, F_2, \ldots, F_9\}$ is a $9$-cycle system of $\G_3+4$.
Therefore, $\mathcal{U}_1 = Orb(A)\ \cup\ Orb(B)\ \cup\ \{F_1, F_2, \ldots, F_9\}$ provides a
$9$-cycle system of $\Omega_2 + 9$.
Since the set $\{C[2]\mid C\in\mathcal{U}_1\}$ decomposes $(\Omega_2 + 9)[2]$,
and each $C[2]$ decomposes into two $9$-suns, we can easily obtain a
$9$-sun system of $(\Omega_2 + 9)[2]$. Indeed, letting
\begin{align*}
  & S_1 = \sigma(x_1, \ov{x_2}, y_3, y_4, a_5, \ldots, a_9),\;\;
    S_3 = \sigma(y_1, \ov{y_2}, b_3, \ldots, b_9),\\
  & S_{3+2j} = \sigma(c_j,\ov{c_{j+1}}, f_{j,3}, f_{j,4}, \ldots, f_{j,9}) \;\;
  \text{for $j\in[1,9]$},
    \;\; \text{and}\;\;\\
  & S_{2i} = \ov{S_{2i-1}}\;\; \text{for $i\in[1,11]$},
\end{align*}
we have that $A[2]=S_1 \oplus S_2$, $B[2]=S_3 \oplus S_4$, and
$F_j[2] = S_{3+2j}\oplus S_{4+2j}$, for every $j\in[1,9]$. Therefore
$\mathcal{U}_2 = \bigcup_{i=1}^4 Orb(S_i) \cup \{S_5, S_6, \ldots, S_{22}\}$
is a $9$-sun system of $\Omega_2[2] + 18$.

We finally use $\mathcal{U}_2$ to build a $9$-sun system of $\Omega^*_2[2] + 21 = (\Omega_2[2] + 18) \oplus (I + 3)$. By replacing the vertices of each $S_i$, as outlined in
Table \ref{Table}, we obtain the $9$-sun $T_i$.
The new 22 graphs, $T_1, T_2, \ldots, T_{22}$, are built in such a way that
\begin{align*}
  (a)\; & \text{$\bigcup_{i=1}^4 Orb(T_i) \cup \{T_5, T_6, \ldots, T_{22}\}$
        decomposes a subgraph $K$ of $\Omega^*_2[2] + 21$};\\
  (b)\; &\text{$(\Omega^*_2[2] + 21)\setminus{K}$ decomposes into seven
  $9$-suns.}
\end{align*}
This way we obtain a $9$-sun system of $\Omega^*_2[2] + 21$, and hence
the desired $9$-sun system of $K_{36}+21$.
\end{ex}

\begin{thm}\label{thm:holek=3}
Let $k\equiv 3\pmod 4\geq 7$ and $n\equiv 0,1 \pmod 4$ with  $2k < n< 10k$.
If $n\not\equiv 2,3 \pmod{k-1}$ and $\left\lfloor\frac{n-4}{k-1}\right\rfloor$ is even, then there exists a $k$-sun system of $K_{4k} + n$
except possibly when $(k,n)\in \{(7,64),(7,65)\}$.
\end{thm}
\begin{proof}
First,  $k\equiv 3\pmod 4\geq 7$ implies that $\ell\geq 3$ is odd.
Now, let $n = 2(q\ell + r) + \nu$ with $1\leq r\leq \ell$ and $\nu\in\{2,3\}$.
Note that $q=\left\lfloor\frac{n-4}{k-1}\right\rfloor$, hence $q$ is even.
Also, since  $2k < n< 10k$, we have $2\leq q \leq 10$.
By $q$ even and $n\equiv 0,1 \pmod 4$ it follows that $r$ is odd,
and $n\not\equiv 2,3 \pmod{k-1}$ implies that $r\neq \ell$.
To sum up,
\begin{equation}
\nonumber  \text{$q$ is even with $2\leq q \leq 10$, and $r$ is odd with $1\leq r\leq \ell-2$}.
\end{equation}
As  in the previous theorem,
let $V(K_{4k}+n) = \big(\Z_{k}\times[0,3]\big)\cup\{\infty_h\mid h\in \Z_{n-\nu}\}
\cup \{\infty'_1,\infty'_2, \infty'_\nu\}$.

We split the proof into two cases.

Case 1) $q\leq 2r+4$.
We start decomposing $K_{2k}$ into the following two graphs:
\begin{align*}
  \G_1 = \big\langle [3,\ell], [k-2r-2, k], [3,\ell] \big\rangle \;\;\mbox{and}\;\;
  \G_2 = \big\langle \{1,2\}, [1,k-2r-3], \{1,2\} \big\rangle.
\end{align*}
Since $q\leq 2r+4$, the graph $\G_1$ can be further decomposed into the following graphs:
 \[
 \G_{1,1} =
  \big\langle \{\ell\},[k-2r+q-3,k], \varnothing \big\rangle, \hspace{.3cm}
 \G_{1,2} =
 \big\langle [3,\ell-1], \varnothing,[3,\ell]\big\rangle,\]
 \[
  \G_{1,3} =
 \big\langle \varnothing,[k-2r-2,k-2r+q-4], \varnothing\big\rangle.\]
The first two graphs have a $k$-cycle system
 by Lemmas \ref{l, S S+1, 0} and \ref{ab, 0, 0}, while $\G_{1,3}$ decomposes into $(q-1)$ $1$-factors, say $J_1,J_2,\ldots,J_{q-1}$.
Setting $w_1=(q-1)\ell$, by Lemma \ref{trivial} we have that:
\[ \G_1 + (q-1)\ell = \oplus_{i=1}^{q-1} (J_i + \ell) \oplus (\G_{1,1}\oplus \G_{1,2}).
\]
Hence $\G_1+(q-1)\ell$ has a $k$-cycle system
since each $J_i+\ell$ decomposes into $k$-cycles by Lemma \ref{+l}.

Letting $w_2=n-2w_1=2(\ell+r)+\nu$ and recalling that $K_{2k}=\G_1\oplus \G_2$, by Lemma \ref{lemma:main}
it remains  to construct a $k$-sun system of $\G_2^*[2] + w_2.$
We start decomposing $\G_2$ into the following graphs:
\begin{align*}
  \G_{2,0} &= \big\langle \{1,2\}, [1,k-2r-4], \{1,2\}  \big\rangle  \;\; \text{and}\;\;
  \G_{2,1} = \big\langle \varnothing, \{k-2r-3\}, \varnothing \big\rangle.
\end{align*}
Recalling that $\G^*_2[2]= \G_2[2] \oplus I$,
where $I$ denotes the $1$-factor $\big\{ \{z, \ov{z}\} \mid z\in \Z_k \times\{0,1\}\big\}$ of $K_{4k}$,
by Lemma \ref{trivial} we have that
\begin{equation}\nonumber
\G_2^*[2] + w_2 =  (\G_{2,1} + \ell )[2] \oplus (\G_{2,0} + r)[2] \oplus (I+\nu).
\end{equation}
By Lemmas \ref{+r} and \ref{+l} there exist a $k$-cycle $A=(x_1,x_2,x_3,y_4,y_5,y_6,a_7,\ldots,a_k)$ of $\G_{2,0} + r$
and a $k$-cycle $B=(y,x,b_3,\ldots,b_{k})$ of
$\G_{2,1} + \ell$, satisfying the following properties:
\begin{align*}
    & 
    \text{$Orb(A)\ \cup\ Orb(B)$ is a $k$-cycle system of $\G_2 +(\ell+r)$;} \\
    & 
    \text{$Dev(\{x_1, x_2\})$ and $Dev(\{x_2, x_3\})$ are $k$-cycles with vertices in $\Z_{k}\times\{0\}$;}\\
    & 
    \text{$Dev(\{y_4, y_5\})$ and $Dev(\{y_5, y_6\})$ are $k$-cycles with vertices in $\Z_{k}\times\{1\}$;}\\
     & 
     \text{$x\in \Z_{k}\times\{0\}$ and $y\in \Z_{k}\times\{1\}$.}
\end{align*}
Set $A'=(x_1, \ov{x_2}, x_3,y_4,\ov{y_5},y_6,a_7, \ldots, a_k)$
and $B'=(y,\ov{x},b_3,\ldots,b_{k})$
 and let $\cS = \{\sigma(A'), \ov{\sigma(A')},\sigma(B'), \ov{\sigma(B')}\}$.
By Lemma \ref{from k-C to k-S}, we have that
$\bigcup_{S\in \cS} Orb(S)$ is a $k$-sun system of
$\big(\G_2 +(\ell+r)\big)[2] = \G_2[2] +2(\ell+r) = (\G_2^*[2] + w_2)\setminus (I+\nu)$.

To construct a $k$-sun system of $\G_2^*[2] + w_2$ we proceed as in Theorem \ref{thm:holek=1}.
We modify the graphs in $\cS$ and obtain four $k$-suns $T_1,T_2,T_3,T_4$
whose translates between them cover all edges incident with
$\infty'_1, \infty'_2$, and possibly $\infty'_3$ when $\nu=3$.
Then we construct further $2\nu+1$ $k$-suns $G_1,\ldots,G_{2\nu+1}$ to cover the missing edges.
The reader can check that $\bigcup_{i=1}^4 Orb(T_i) \cup \{G_1,\ldots,G_{2\nu+1}\}$ is a
$k$-sun system of $\G_2^*[2] + w_2$.

%
%
The graphs $T_i$ are the following, where the elements in bold are the replaced vertices:
\begingroup
\allowdisplaybreaks
\begin{align*}
T_1 &=
\begin{cases}
  \left(
  \begin{matrix}
 x_1  &  \ov{x_2}  &  x_3  &  \bm{\infty'_2} &  \ov{y_5} &  y_6  &  a_7      & \ldots & a_{k-1}     & a_k\\
        \bm{\infty'_1}  &  \ov{x_3} &  \ov{y_4}  & y_5 &  \bm{y_4} &  \ov{a_7}  &  \ov{a_8} & \ldots & \ov{a_{k}}  & \ov{x_1}
  \end{matrix}
  \right) &
  \mbox{if $\nu=2$},\\
  \left(
  \begin{matrix}
        x_1  &  \ov{x_2}  &  x_3  &  \bm{\infty'_2} &  \ov{y_5} &  y_6  &  a_7      & \ldots & a_{k-1}     & a_k\\
  \bm{\infty'_1}  &  \bm{\infty'_3} &  \ov{y_4}  & y_5 &  \bm{y_4} &  \ov{a_7}  &  \ov{a_8} & \ldots & \ov{a_{k}}  & \ov{x_1}
  \end{matrix}
  \right) &
  \mbox{if $\nu=3$},
\end{cases}\\
T_2 &=
\begin{cases}
  \left(
  \begin{matrix}
\ov{x_1}  &  x_2  &  \ov{x_3}  &  \bm{\infty'_1} &  y_5 &  \ov{y_6}  &  \ov{a_7}      & \ldots & \ov{a_{k-1}}     & \ov{a_k}\\
        \bm{\infty'_2}  &  x_3 &  y_4  & \ov{y_5} &  y_6 &  a_7  &  a_8 & \ldots & a_{k}  & x_1
  \end{matrix}
  \right) &
  \mbox{if $\nu=2$},\\
  \left(
  \begin{matrix}
        \ov{x_1}  &  x_2  &  \ov{x_3}  &  \bm{\infty'_1} &  y_5 &  \ov{y_6}  &  \ov{a_7}      & \ldots & \ov{a_{k-1}}     & \ov{a_k}\\
  \bm{\infty'_2}  &  \bm{\infty'_3} &  y_4  & \ov{y_5} &  y_6 &  a_7  &  a_8 & \ldots & a_{k}  & x_1
  \end{matrix}
  \right) &
  \mbox{if $\nu=3$},
\end{cases}\\
T_3 &=
\begin{cases}
  \sigma(B') &
  \mbox{if $\nu=2$},\\
  \left(
  \begin{matrix}
 y &  \ov{x}  &  b_3  &  b_4 &        \ldots & b_{k-1}     & b_{k}\\
        \bm{\infty'_3}  &  \ov{b_3} &  \ov{b_4}  & \ov{b_5} &   \ldots & \ov{b_{k}}  & \ov{y}
  \end{matrix}
  \right) &
  \mbox{if $\nu=3$},
\end{cases}\\
T_4 &=
\begin{cases}
  \ov{\sigma(B')} &
  \mbox{if $\nu=2$},\\
  \left(
  \begin{matrix}
     \ov{y}  &  x  &  \ov{b_3}  &    \ov{b_4}      & \ldots & \ov{b_{k-1}}     & \ov{b_{k}}\\
  \bm{\infty'_3}  &  b_3 &  b_4  & b_5 &   \ldots & b_{k}  & y
  \end{matrix}
  \right) &
  \mbox{if $\nu=3$}.
\end{cases}
\end{align*}
\endgroup

%

The graphs $G_i$, for $i = [1,2\nu+1]$, are so defined:
\[
\begin{array}{ll}
\vspace{0.2cm}
    G_1 = Dev\big(x_1 \sim x_2 \sim \ov{x_2}\big),
    &
   G_2 = Dev(y_5 \sim y_4 \sim x_3),
\\[1ex] \rule{0pt}{1\normalbaselineskip}
      G_3 = Dev(\{\ov{x_1}, \ov{x_2}\}\oplus \{\ov{x_3},\ov{y_4}\}),
      &
  G_4 = Dev(\ov{y_5} \sim \ov{y_4} \sim y_5),
  \\[1ex] \rule{0pt}{1\normalbaselineskip}
  G_5 = Dev(\ov{y_5} \sim \ov{y_6} \sim y_6),
  &
      G_6 = Dev(\{x_2, x_3\}\oplus \{x,y\}),
      \\[1ex] \rule{0pt}{1\normalbaselineskip}
   G_7 = Dev(\{\ov{x_2}, \ov{x_3}\}\oplus \{\ov{x},\ov{y}\}).
   &
  \end{array}
  \]

Case 2) $q\geq 2r+6$. Note that this implies $r=1$ and $q=8,10$. As before $K_{2k}=\G_1\oplus \G_2$ where
\begin{align*}
  \G_1 = \big\langle [3,\ell], \{0\}\cup [k-5, k-1], [3,\ell] \big\rangle \;\;\mbox{and}\;\;
  \G_2 = \big\langle \{1,2\}, [1,k-6], \{1,2\} \big\rangle.
\end{align*}
Since $(k,n)\neq (7,64),(7,65)$ then $(\ell,q)\neq(3,10)$, hence
 the graph $\G_1$ can be  decomposed into the following graphs:
\[
 \G_{1,1} =
 \big\langle \varnothing,[k-5,k-1], \varnothing\big\rangle, \quad\quad
 \G_{1,2} =
  \left\langle \left[3,    \frac{q-2}{2}\right], \{0\},\left[3,\frac{q-2}{2}\right] \right\rangle,\]
  \[
 \G_{1,3} =
 \left\langle \left[\frac{q}{2},\ell\right], \varnothing,\left[\frac{q}{2},\ell\right]\right\rangle.
\]
The graph $\G_{1,1}$ decomposes into five $1$-factors $J_1,\ldots,J_5$, while by Lemma \ref{HLR:1factorization} $\G_{1,2}$ decomposes
into $(q-5)$ $1$-factors $J'_1,\ldots, J'_{q-5}$.
Letting $w_1=q\ell$, by Lemma \ref{trivial} we have that
\[
\G_1+w_1=(\G_{1,1}+5\ell)\oplus (\G_{1,2}+(q-5)\ell)\oplus \G_{1,3}=\oplus_{i=1}^5(J_i+\ell) \oplus \left[\oplus_{i=1}^{q-5}(J'_i+\ell)\right]\oplus \G_{1,3}.
\]
By Lemmas \ref{+l} and \ref{ab, 0, 0}, each $J_i+\ell$, each $J'_i+\ell$ and $\G_{1,3}$ decompose into $k$-cycles.
Hence $\G_1 + q\ell$ has a $k$-cycle system.
Let now $w_2=n-2w_1=2+\nu$. Note that a $k$-sun system of $\G_2^*[2]+w_2$ can be obtained as in Case 1, where $\G_{2,1}$ is empty.
\end{proof}

\begin{thm}\label{thm:holek=3,a}
Let $k\equiv 3\pmod 4\geq 11$ and $n\equiv 0,1 \pmod 4$ with  $2k < n < 10k$.
If $\left\lfloor\frac{n-4}{k-1}\right\rfloor$ is even, and $n\equiv 2,3 \pmod{k-1}$, then there is a
$k$-sun system of $K_{4k}+n$, except possibly when $(k,n)\in\{(11,112),$ $(11,113)\}$.
\end{thm}
\begin{proof}
  Let $n = 2(q\ell + r) + \nu$ with $1\leq r\leq \ell$ and $\nu\in\{2,3\}$. Clearly,
$q = \left\lfloor\frac{n-4}{k-1}\right\rfloor$, hence $q$ is even. Since $k\geq 11$, $2k <n < 10k$ and $n\equiv 2,3 \pmod{2\ell}$,
 we have that
\begin{eqnarray*}
\text{$q$ is even with $2\leq q\leq 10$ and $r=\ell\geq 5$ is odd}.
\end{eqnarray*}
As before,
let $V(K_{4k}+n) = \big(\Z_{k}\times[0,3]\big)\cup\{\infty_h\mid h\in \Z_{n-\nu}\}
\cup \{\infty'_1,\infty'_2, \infty'_\nu\}$.

We start decomposing $K_{2k}$ into the following two graphs:
\[  \G_1 = \big\langle [3,\ell],[k-3,k] , [4,\ell]  \big\rangle,\hspace{.3cm}
  \G_2 = \big\langle \{1,2\}, [1, k-4], \{1,2,3\} \big\rangle.
\]
If $q=2,4$, $\G_1$ can be further decomposed into
\[ \G_{1,1} = \big\langle \varnothing,[k-3,k-4+q], \varnothing \big\rangle,\hspace{.3cm}
\G_{1,2} = \big\langle \varnothing,[k-3+q,k],\{\ell\} \big\rangle,
\]
\[
  \G_{1,3} = \big\langle [3,\ell],\varnothing, [4,\ell-1]\big\rangle.
\]
The graph $\G_{1,1}$ decomposes into $q$ $1$-factors, say $J_1,\ldots,J_q$.
Letting $w_1=q\ell$, by Lemma \ref{trivial} we have that
\[
\G_1+w_1=(\G_{1,1}+w_1)\oplus \G_{1,2} \oplus \G_{1,3}= \oplus_{i=1}^{q}(J_i+\ell)\oplus \G_{1,2} \oplus \G_{1,3} .
\]
Lemmas \ref{+l}, \ref{l, S S+1, 0} and \ref{ab, 0, 0} guarantee that each $J_i+\ell$, $\G_{1,2}$ and $\G_{1,3}$
decompose into $k$-cycles, hence $\G_1+w_1$ has a $k$-cycle system.
Suppose now $q\geq 6$. By $(k,n)\not\in\{(11,112),(11,113)\}$,
we have $(\ell,q)\neq (5,10)$.
 In this case $\G_1$ can be further decomposed into
\[ \G_{1,1} = \big\langle \varnothing,[k-3,k-1], \varnothing \big\rangle,\hspace{.3cm}
\G_{1,2} = \left\langle \left[\ell+3 - \frac{q}{2}, \ell\right],\{0\},\left[\ell+3 - \frac{q}{2}, \ell\right] \right\rangle,
\]
\[
\G_{1,3}=\left\langle \left[3, \ell+2 - \frac{q}{2}\right],\varnothing,
             \left[4, \ell+2 - \frac{q}{2}\right] \right\rangle.
\]
The graph $\G_{1,1}$ can be decomposed into three $1$-factors say $J_1,J_2,J_3$, also by Lemma \ref{HLR:1factorization}
the graph $\G_{1,2}$ can be decomposed into $(q-3)$ $1$-factors say $J'_1,\ldots,J'_{q-3}$.
Set again $w_1=q\ell$, by Lemma \ref{trivial} we have that
\[
\G_1+w_1=(\G_{1,1}+3\ell)\oplus (\G_{1,2}+(q-3)\ell) \oplus \G_{1,3} = \oplus_{i=1}^{3}(J_i+\ell)\oplus \left[\oplus_{j=1}^{q-3}(J'_j+\ell)\right] \oplus \G_{1,3}.
\]
By Lemmas \ref{+l} and \ref{ab, 0, 0} we have that each $J_i+\ell$, each $J'_j+\ell$ and $\G_{1,3}$  decompose into $k$-cycles,
hence $\G_1+w_1$ has a $k$-cycle system.
Hence for any value of $q$ we have proved that $\G_1+w_1$ has a $k$-cycle system.

Now, setting $w_2 = n - 2w_1 = 2\ell + \nu$ and recalling that $K_{2k} = \G_1 \oplus \G_2$,
by Lemma \ref{lemma:main} it is left to show that
$\G_2^*[2] + w_2$ has a $k$-sun system.
Let $r_1$ and $r_2\geq 2$ be an odd and an even integer, respectively, such that $r_1+r_2=r=\ell$.
Note that $\G_2$ can be further decomposed into
\[
  \G_{2,1} = \big\langle \{1\}, [1, k-2r_1-2], \{1\} \big\rangle, \hspace{.3cm}
  \G_{2,2} = \big\langle \{2\},[k-2r_1-1,k-4], \{2,3\}\big\rangle.
\]
Recalling that $\G^*_2[2]= \G_2[2] \oplus I$,
where $I$ denotes the $1$-factor $\big\{ \{z, \ov{z}\} \mid z\in \Z_k \times\{0,1\}\big\}$ of $K_{4k}$,
by Lemma \ref{trivial} we have that
\[
\G_2^*[2] + w_2 = \oplus_{i=1}^2 \big(\G_{2,i} + r_i \big)[2] \oplus (I+\nu).
\]

By Lemma \ref{+r} there is a $k$-cycle $A=(y_1,y_2,x_3,x_4,a_5,\ldots, a_k)$ of $\G_{2,1}+r_1$
and a $k$-cycle $B=(x_1,x_2,y_3,y_4,b_5,\ldots,b_k)$ of  $\G_{2,2}+r_2$
such that
\begin{align}
  \begin{split}\label{thm:holek=3,a2}
     & \text{$Orb(A) \ \cup\ Orb(B)$ is a $k$-cycle system of $\G_2 +\ell$,}\\
    & \text{$Dev(\{x_1,x_2\})$ and $Dev(\{x_3,x_4\})$ are $k$-cycles with vertices in $\Z_k\times \{0\}$,} \\
    & \text{$Dev(\{y_1,y_2\})$ and $Dev(\{y_3,y_4\})$ are $k$-cycles with vertices in $\Z_k\times \{1\}$.}
  \end{split}
\end{align}
Set $A'=(y_1,\ov{y_2},x_3,\ov{x_4},a_5,\ldots, a_k)$
and $B'=(x_1,\ov{x_2},y_3,\ov{y_4},b_5,\ldots,b_k)$.
Let $\cS=\{\sigma(A'),\ov{\sigma(A')},\sigma(B'),\ov{\sigma(B')}\}$, by Lemma \ref{from k-C to k-S},
we have that  $\bigcup_{S\in \cS} Orb(S)$ is a $k$-sun system of
$\big(\G_2 +\ell\big)[2] = \G_2[2] +2\ell = (\G_2^*[2] + w_2)\setminus (I+\nu)$.
To construct a $k$-sun system of $\G_2^*[2] + w_2$, we build a family $\cT=\{T_1,T_2,T_3, T_4\}$ of $k$-suns
by modifying
the graphs in $\cS$ so that $\bigcup_{T\in \cT} Orb(T)$ covers all the edges incident with
$\infty'_1, \infty'_2$, and possibly $\infty'_3$ when $\nu=3$.
We then construct further
$(2\nu+1)$ $k$-suns $G_1, G_2, \ldots, G_{2\nu +1}$ which cover the remaining edges exactly once. Hence,
$\bigcup_{T\in \cT} Orb(T) \cup \{G_1, G_2, \ldots, G_{2\nu +1}\}$ is a
$k$-sun system of $\G_2^*[2] + w_2$.

The graphs $T_1, \ldots, T_4$ and $G_1, \ldots, G_{2\nu +1}$
are the following, where as before the elements in bold are the replaced vertices.
\begingroup
\allowdisplaybreaks
\begin{align*}
T_{1} &=
  \left(
  \begin{matrix}
         y_1        &    \ov{y_2}    &       x_3  &   \ov{x_4} & a_5  &  \ldots &     a_{k-1}    & a_k\\
       \bm{\infty'_2} & \ov{x_3}           & x_4   &   \ov{a_5} & \ov{a_6} & \ldots & \ov{a_k}   & \ov{y_1}\\
  \end{matrix}
  \right),\\
T_2 &=
\begin{cases}
  \left(
  \begin{matrix}
 \ov{y_1} &      \bm{\infty'_1}  &  \ov{x_3} &        x_4 &  \ov{a_5}    & \ldots & \ov{a_{k-1}}  & \ov{a_k}\\
    \bm{\infty'_2} & x_3  &  \ov{x_4} &     a_5  &       a_6 & \ldots &     a_{k}    & \ov{y_1}
  \end{matrix}
  \right) &
  \mbox{if $\nu=2$},\\
  \left(
  \begin{matrix}
   \ov{y_1} &      \bm{\infty'_1}  &  \ov{x_3} &        x_4 &  \ov{a_5}    & \ldots & \ov{a_{k-1}}  & \ov{a_k}\\
    \bm{\infty'_2} & x_3  &  \bm{\infty'_3} &     a_5  &       a_6 & \ldots &     a_{k}    & \ov{y_1}
  \end{matrix}
  \right) &
  \mbox{if $\nu=3$},
\end{cases}\\
T_{3} &=
  \left(
  \begin{matrix}
         x_{1} &    \ov{x_2}    &     y_3       &    \ov{y_4}  & b_5   &     \ldots &    b_{k-1}    & b_{k}\\
       \bm{\infty'_2} & \ov{y_3}  & \bm{\infty'_1} &  \ov{b_5} & \ov{b_6} &    \ldots & \ov{b_k}   & \ov{x_{1}}\\
  \end{matrix}
  \right),\\
T_4 &=
\begin{cases}
  \left(
  \begin{matrix}
 \ov{x_1} &      x_2  &  \ov{y_3} &        y_4 &  \ov{b_5}   & \ldots & \ov{b_{k-1}}  & \ov{b_k}\\
    \bm{\infty'_2} & y_3  &  \ov{y_4} &     b_5  &       b_6  & \ldots &     b_{k}    & {x_1}
  \end{matrix}
  \right) &
  \mbox{if $\nu=2$},\\
  \left(
  \begin{matrix}
\ov{x_1} &      x_2  &  \bm{\infty'_3} &        y_4 &  \ov{b_5}   & \ldots & \ov{b_{k-1}}  & \ov{b_k}\\
    \bm{\infty'_2} & y_3  &  \ov{y_4} &   b_5  &       b_6  & \ldots &     b_{k}    & {x_1}
  \end{matrix}
  \right) &
  \mbox{if $\nu=3$}.
\end{cases}
\end{align*}
\endgroup
\[
\begin{array}{ll}
  G_1 = Dev\big(y_1 \sim y_2 \sim x_3 \big),
&
  G_2 = Dev\big(\ov{y_2} \sim \ov{y_1} \sim y_2 \big),
\\[1ex] \rule{0pt}{1\normalbaselineskip}
  G_3 = Dev\big(y_3 \sim y_4 \sim \ov{y_4} \big),
&
  G_4 = Dev\big(\{\ov{x_1}, \ov{x_2}\} \oplus \{\ov{x_3}, y_2\} \big),
\\[1ex] \rule{0pt}{1\normalbaselineskip}
   G_5 =
  \begin{cases}
     Dev\big( x_1 \sim x_2 \sim \ov{x_2} \big) & \text{if $\nu=2$},\\
     Dev\big( x_1 \sim x_2 \sim \ov{y_3} \big) & \text{if $\nu=3$},\\
  \end{cases}
&
  G_6 = Dev\big(\ov{x_3} \sim \ov{x_4} \sim x_4\} \big),
\\[1ex] \rule{0pt}{1\normalbaselineskip}
  G_7 = Dev\big(\ov{y_4} \sim \ov{y_3} \sim y_4 \big).
\end{array}
\]
By recalling (\ref{thm:holek=3,a2}), it is not difficult to check that
the graphs $G_h$ are $k$-suns.

\end{proof}

\begin{thm}\label{thm:holek=3, b1}
Let $k\equiv 3\pmod 4\geq 7$ and $n\equiv 0,1 \pmod 4$ with  $2k < n<10k$.
If $\left\lfloor\frac{n-4}{k-1}\right\rfloor$ is odd and $n\not\equiv 0,1 \pmod{k-1}$, then there is a
$k$-sun system of $K_{4k}+n$.
\end{thm}
\begin{proof} Let $n = 2(q\ell + r) + \nu$ with $1\leq r\leq \ell$ and $\nu\in\{2,3\}$. Clearly,
$q = \lfloor\frac{n-4}{k-1}\rfloor$.
Also, we have that $q$ and $\ell\geq 3$ are odd,
and $n\equiv 0,1 \pmod 4$; hence $r$ is even.
Furthermore, we have that $2\leq q \leq 10$, since by assumption
$2k < n< 10k$.  Considering now the hypothesis that
$n\not\equiv 0,1 \pmod{2\ell}$, it follows that $r\neq \ell-1$. To sum up,
\begin{equation}\label{thm:holek=3,b1:cond1}
  \text{$q$ is odd with $3\leq q \leq 9$, and $r$ is even with $2\leq r\leq \ell-3$}.
\end{equation}
As before,
let $V(K_{4k}+n) = \big(\Z_{k}\times[0,3]\big)\cup\{\infty_h\mid h\in \Z_{n-\nu}\}
\cup \{\infty'_1,\infty'_2, \infty'_\nu\}$.

We start decomposing $K_{2k}$ into the following two graphs:
\begin{align*}
  \G_1 = \big\langle [4,\ell], [k-2r-1, k], [3,\ell] \big\rangle \;\;\mbox{and}\;\;
  \G_2 = \big\langle [1,3], [1,k-2r-2], [1,2]  \big\rangle.
\end{align*}
Considering that $3\leq q\leq9\leq 2r+5$, the graph $\G_1$ can be further decomposed into the following graphs:
\[
  \G_{1,1} =
  \big\langle [4,\ell], \varnothing, [3,\ell-1]  \big\rangle,  \hspace{.3cm}
  \G_{1,2} =
  \big\langle \varnothing, [k-2r-4 + q, k], \{ \ell \} \big\rangle,
\]
\[
  \text{and}\;\; \G_{1,3} =
  \big\langle  \varnothing, [k-2r-1, k-2r-5+q], \varnothing \big\rangle.
\]
The first two have a $k$-cycle system by Lemmas \ref{ab, 0, 0} and \ref{l, S S+1, 0}, while
$ \G_{1,3}$ decomposes into $(q-3)$ $1$-factors, say $J_1, J_2, \ldots, J_{q-3}$.
Letting $w_1 = (q-3)\ell$,
by Lemma \ref{trivial} we have that
\[
\G_1 + w_1 = \oplus_{i=1}^{q-3} (J_i + \ell) \oplus (\G_{1,1}\oplus \G_{1,2}).
\]
Therefore, $\G_1 + w_1$ has a $k$-cycle system, since each $J_i + \ell$ decomposes into $k$-cycles
by Lemma \ref{+l}. Setting $w_2 = n - 2w_1 = 2(3\ell+r) + \nu$ and recalling that $K_{2k} = \G_1 \oplus \G_2$,
by Lemma \ref{lemma:main} it is left to show that
$\G_2^*[2] + w_2$ has a $k$-sun system.

We start decomposing $\G_2$ into the following graphs:
\begin{align*}
  \G_{2,0} &= \big\langle [1,3], [1,k-2r-5], [1,2]  \big\rangle,  \;\; \text{and} \\
  \G_{2,i} &= \big\langle \varnothing, \{k-2r-5+i\}, \varnothing \big\rangle, \;\; \text{for $1\leq i\leq 3$}.
\end{align*}
Recalling that $\G^*_2[2]= \G_2[2] \oplus I$,
where $I$ denotes the $1$-factor $\big\{ \{z, \ov{z}\} \mid z\in \Z_k \times\{0,1\}\big\}$ of $K_{4k}$,
by Lemma \ref{trivial} we have that
\[
\G_2^*[2] + w_2 = \oplus_{i=1}^3 \big(\G_{2,i} + \ell \big)[2] \oplus (\G_{2,0} + r)[2] \oplus (I+\nu).
\]
By Lemmas \ref{+r} and \ref{+l} there exist a $k$-cycle $A=(x_1,x_2,x_3,y_4,y_5,y_6,a_7,\ldots,a_k)$ of $\G_{2,0} + r$,
a $k$-cycle $B_1=(x_{1,0},y_{1,1}, b_{1,2}, \ldots, b_{1,k-1})$ of $\G_{2,1} + \ell$,
and a $k$-cycle $B_i=(y_{i,0},x_{i,1}, b_{i,2}, \ldots, b_{i,k-1})$ of
$\G_{2,i} + \ell$, for $2\leq i\leq 3$, satisfying the following properties:

\begin{align}
  \begin{split}\label{thm:holek=3,b1:cond2}
    & \text{$Dev(\{x_1,x_2\})$ and $Dev(\{x_2,x_3\})$ are $k$-cycles with vertices in $\Z_k\times \{0\}$,} \\
    & \text{$Dev(\{y_4,y_5\})$ and $Dev(\{y_5,y_6\})$ are $k$-cycles with vertices in $\Z_k\times \{1\}$;}\\
  \end{split} \\
    \begin{split} \label{thm:holek=3,b1:cond3}
    & \text{$x_{1,0}, x_{2,1}, x_{3,1}\in  \Z_{k}\times \{0\}$,
      $y_{1,1}, y_{2,0}, y_{3,0}\in \Z_{k}\times \{1\}$}; \\
    \end{split}\\
    & \label{thm:holek=3,b1:cond4}
    \text{$\bigcup_{i=1}^3 Orb(B_i) \ \cup\ Orb(A)$ is a $k$-cycle system of $\G_2 +(3\ell+r)$.}
\end{align}

Set $A'=(x_1, \ov{x_2}, x_3, \ov{y_4}, y_5, \ov{y_6}, a_7, a_8, \ldots, a_{k-1}, a_k)$ and
let $\cS = \{\sigma(A'), \ov{\sigma(A')}\}\ \cup\
\{\sigma(B_i), \ov{\sigma(B_i)} \mid 1\leq i \leq 3\}$.
By Lemma \ref{from k-C to k-S}, we have that
$\bigcup_{S\in \cS} Orb(S)$ is a $k$-sun system of
$\big(\G_2 +(3\ell+r)\big)[2] = \G_2[2] +2(3\ell+r) = (\G_2^*[2] + w_2)\setminus (I+\nu)$.

To construct a $k$-sun system of $\G_2^*[2] + w_2$, we build a family $\cT=\{T_0,T_1,\ldots, T_7\}$ of $k$-suns
by modifying
the graphs in $\cS$ so that $\bigcup_{T\in \cT} Orb(T)$ covers all the edges incident with
$\infty'_1, \infty'_2$, and possibly $\infty'_3$ when $\nu=3$.
We then construct further
$(2\nu+1)$ $k$-suns $G_1, G_2, \ldots, G_{2\nu +1}$ which cover the remaining edges exactly once. Hence,
$\bigcup_{T\in \cT} Orb(T) \cup \{G_1, G_2, \ldots, G_{2\nu +1}\}$ is a
$k$-sun system of $\G_2^*[2] + w_2$.

The graphs $T_0, \ldots, T_7$ and $G_1, \ldots, G_{2\nu +1}$
are the following, where as before the elements in bold are the replaced vertices.
\begingroup
\allowdisplaybreaks
\begin{align*}
T_0 &=
\begin{cases}
  \left(
  \begin{matrix}
  x_1  &  \ov{x_2}  &  x_3  &  \ov{y_4} &  y_5 &  \ov{y_6}  &  a_7      & \ldots & a_{k-1}     & a_k\\
        x_2  &  \bm{\infty'_1} &  y_4  & \bm{\infty'_2} &  y_6 &  \ov{a_7}  &  \ov{a_8} & \ldots & \ov{a_{k}}  & \ov{x_1}
  \end{matrix}
  \right) &
  \mbox{if $\nu=2$},\\
  \left(
  \begin{matrix}
       x_1  &  \ov{x_2}  &  x_3  &  \ov{y_4} &  y_5 &  \ov{y_6}  &  a_7      & \ldots & a_{k-1}     & a_k\\
  \bm{\infty'_3}  &  \bm{\infty'_1} &  y_4  & \bm{\infty'_2} &  y_6 &  \ov{a_7}  &  \ov{a_8} & \ldots & \ov{a_{k}}  & \ov{x_1}
  \end{matrix}
  \right) &
  \mbox{if $\nu=3$},
\end{cases}\\
T_1 &=
\begin{cases}
  \left(
  \begin{matrix}
 \ov{x_1} &       x_2  &  \ov{x_3} &        y_4 &  \ov{y_5} &  y_6  &  \ov{a_7} & \ldots & \ov{a_{k-1}}  & \ov{a_k}\\
    \ov{x_2} & \bm{\infty'_1}  &  \ov{y_4} &  \bm{\infty'_2} &      \bm{y_5}  &  a_7  &       a_8 & \ldots &     a_{k}    & x_1
  \end{matrix}
  \right) &
  \mbox{if $\nu=2$},\\
  \left(
  \begin{matrix}
   \ov{x_1} &       x_2  &  \ov{x_3} &       y_4 &  \ov{y_5} &  y_6  &  \ov{a_7} & \ldots & \ov{a_{k-1}}  & \ov{a_k}\\
   \bm{\infty'_3} & \bm{\infty'_1}  &  \ov{y_4} &  \bm{\infty'_2} &    \bm{y_5}  &  a_7  &       a_8 & \ldots &     a_{k}    & x_1
  \end{matrix}
  \right) &
  \mbox{if $\nu=3$},
\end{cases}\\
T_{2} &=
  \left(
  \begin{matrix}
         x_{1,0} &    y_{1,1}    &     b_{1,2}  &    \ldots &     b_{1,k-2}    & b_{1,k-1}\\
       \bm{\infty'_2} & \ov{b_{1,2}}  & \ov{b_{1,3}} &    \ldots & \ov{b_{1,k-1}}   & \ov{x_{1,0}}\\
  \end{matrix}
  \right),\\
T_{3} &=
 \left(
  \begin{matrix}
        \ov{x_{1,0}} &   \ov{y_{1,1}}    &     \ov{b_{1,2}}  &    \ldots &    \ov{b_{1,k-2}}    & \ov{b_{1,k-1}}\\
       \bm{\infty'_2} & b_{1,2}  & b_{1,3} &    \ldots & b_{1,k-1}   & x_{1,0}\\
  \end{matrix}
  \right),\\
T_{4} &=
  \left(
  \begin{matrix}
         y_{2,0} &    x_{2,1}    &     b_{2,2}  &    \ldots &     b_{2,k-2}    & b_{2,k-1}\\
       \bm{\infty'_1} & \ov{b_{2,2}}  & \ov{b_{2,3}} &    \ldots & \ov{b_{2,k-1}}   & \ov{y_{2,0}}\\
  \end{matrix}
  \right),\\
T_{5} &=
  \left(
  \begin{matrix}
         \ov{y_{2,0}} &   \ov{ x_{2,1}}    &   \ov{  b_{2,2}}  &    \ldots &    \ov{ b_{2,k-2}}    & \ov{b_{2,k-1}}\\
        \bm{\infty'_1} & b_{2,2} & b_{2,3} &    \ldots & b_{2,k-1}   & y_{2,0}\\
  \end{matrix}
  \right),\\
T_6 &=
\begin{cases}
  \sigma(B_3) &
  \mbox{if $\nu=2$},\\
  \left(
  \begin{matrix}
  y_{3,0} &    x_{3,1}      &     b_{3,2}  &    \ldots &     b_{3,k-2}  & b_{3,k-1}\\
    \bm{\infty'_3} &  \ov{b_{3,2}}    & \ov{b_{3,3}} &    \ldots & \ov{b_{3,k-1}}   & \ov{y_{3,0}}\\
  \end{matrix}
  \right) &
  \mbox{if $\nu=3$},
\end{cases}\\
T_7 &=
\begin{cases}
  \ov{\sigma(B_3)} &
  \mbox{if $\nu=2$},\\
  \left(
  \begin{matrix}
  \ov{y_{3,0}} &    \ov{x_{3,1}}      &     \ov{b_{3,2}}  &    \ldots &     \ov{b_{3,k-2}}  & \ov{b_{3,k-1}}\\
  \bm{\infty'_3}  &   b_{3,2}   & b_{3,3} &    \ldots & b_{3,k-1}   & y_{3,0}\\
  \end{matrix}
  \right) &
  \mbox{if $\nu=3$},
\end{cases}
\end{align*}
\endgroup
%
%
\[
\begin{array}{ll}
  G_1 = Dev\big(x_2 \sim x_3 \sim \ov{x_3} \big),
&
  G_2 = Dev\big(\{\ov{x_2}, \ov{x_3}\} \oplus \{\ov{x_{1,0}}, y_{1,1}\} \big),
\\[1ex] \rule{0pt}{1\normalbaselineskip}
  G_3 = Dev\big(\{y_4, y_5\} \oplus \{y_{2,0}, \ov{x_{2,1}}\} \big),
&
  G_4 = Dev\big(\{\ov{y_4}, \ov{y_5}\} \oplus \{\ov{y_{2,0}}, x_{2,1}\} \big),
\\[1ex] \rule{0pt}{1\normalbaselineskip}
  G_5 =  Dev\big(\{\ov{y_5}, \ov{y_6}\} \oplus \{\ov{y_{1,1}}, x_{1,0}\} \big),
&
  G_6 = Dev\big(\{x_1, x_2\} \oplus \{x_{3,1}, \ov{y_{3,0}}\} \big),
\\[1ex] \rule{0pt}{1\normalbaselineskip}
  G_7 = Dev\big(\{\ov{x_1}, \ov{x_2}\} \oplus \{\ov{x_{3,1}}, {y_{3,0}}\} \big).
\end{array}
\]
By recalling \eqref{thm:holek=3,b1:cond2}--\eqref{thm:holek=3,b1:cond4}, it is not difficult to check that the graphs $G_h$ are $k$-suns.
\end{proof}

\begin{thm}\label{thm:holek=3,b2}
Let $k\equiv 3\pmod 4\geq 7$ and $n\equiv 0,1 \pmod 4$ with  $2k < n < 10k$.
If $\left\lfloor\frac{n-4}{k-1}\right\rfloor$ is odd, and $n\equiv 0,1 \pmod{k-1}$, then there is a
$k$-sun system of $K_{4k}+n$ except possibly when $(k, n) \in \{(11, 100), (11, 101)\}$.
\end{thm}
\begin{proof} Let $n = 2(q\ell + r) + \nu$ with $1\leq r\leq \ell$ and $\nu\in\{2,3\}$.
Reasoning as in the proof of Theorem \ref{thm:holek=3, b1} and
considering that $n\equiv 0,1 \pmod{2\ell}$ and $(k, n) \not\in \{(11, 100), (11, 101)\}$,
we have that
\begin{equation}\label{thm:holek=3,b2:cond1}
  \text{$q$ is odd with $3\leq q \leq 9$, $r = \ell-1\geq 2$, $r$ is even, and $(\ell, q)\neq (5,9)$}.
\end{equation}
As before,
let $V(K_{4k}+n) = \big(\Z_{k}\times[0,3]\big)\cup\{\infty_h\mid h\in \Z_{n-\nu}\}
\cup \{\infty'_1,\infty'_2, \infty'_\nu\}$.

We start decomposing $K_{2k}$ into the following two graphs
\[
  \G_1 =
  \left\langle \left[3, \ell\right], \{0\}, \left[3,  \ell\right] \right\rangle,\;\;\mbox{and}\;\;
  \G_2 = \big\langle \{1,2\}, [1, k-1],  \{1,2\} \big\rangle.
\]
Considering \eqref{thm:holek=3,b2:cond1},
we can further decompose $\G_1$ into the following two graphs:
\begin{align*}
  \G_{1,1} &=
  \left\langle \left[3, \frac{q+3}{2}\right], \{0\}, \left[3,  \frac{q+3}{2}\right] \right\rangle,\;\;
  \G_{1,2} =
  \left\langle \left[ \frac{q+5}{2}, \ell\right], \varnothing, \left[ \frac{q+5}{2}, \ell\right] \right\rangle.
\end{align*}
By Lemma \ref{HLR:1factorization}, the graph $\G_{1,1}$ decomposes into $q$ $1$-factors,
say $J_1, J_2, \ldots, J_{q}$. Letting $w_1 = q\ell$, by Lemma \ref{trivial} we have that
\[
\G_1 + w_1 =( \G_{1,1}+w_1) \oplus \G_{1,2} = \oplus_{i=1}^{q} (J_i + \ell) \oplus \G_{1,2}.
\]
Lemmas \ref{+l} and \ref{ab, 0, 0} guarantee that each  $J_i + \ell$ and $\G_{1,2}$ decompose into $k$-cycles,
hence $\G_1 + w_1$ has a $k$-cycle system.
Let $r_1$ and $r_2$ be odd positive integers such that $r=\ell-1 = r_1 + r_2$.
Then, setting $w_2 = n - 2w_1 = 2(r_1 + r_2) + \nu$ and recalling that $K_{2k} = \G_1 \oplus \G_2$,
by Lemma \ref{lemma:main} it is left to show that
$\G_2^*[2] + w_2$ has a $k$-sun system.

We start decomposing $\G_2$ into the following graphs:
\begin{align*}
  \G_{2,1} &= \big\langle \{1\}, [1, k-2r_1-2],  \{1\} \big\rangle \;\;\mbox{and}\;\;
  \G_{2,2}  = \big\langle \{2\}, [k-2r_1-1,k-1], \{2\} \big\rangle.
\end{align*}
Recalling that $\G^*_2[2]= \G_2[2] \oplus I$,
where $I$ denotes the $1$-factor $\big\{ \{z, \ov{z}\} \mid z\in \Z_k \times\{0,1\}\big\}$ of $K_{4k}$,
by Lemma \ref{trivial} we have that
\begin{equation}\label{thm:holek=3,b2:cond2}
\G_2^*[2] + w_2 =  \big(\G_{2,1} + r_1 \big)[2] \oplus (\G_{2,2} + r_2)[2] \oplus (I+\nu).
\end{equation}

By Lemma \ref{+r} there is a $k$-cycle $A= (y_{1}, y_{2}, x_{3}, x_{4}, a_{5}, \ldots, a_{k})$ of $\G_{2,1} + r_1$ and
a $k$-cycle 
$B= (x_{1}, x_{2}, y_{3}, y_{4}, b_{5}, \ldots, b_{k})$ of $\G_{2,2} + r_2$ such that
\begin{align*}
& \text{$Orb(A)\cup Orb(B)$ is a $k$-cycle system of $\G_{2}+r$,}\\
& \text{$Dev(\{x_{3},x_{4}\})$ and $Dev(\{x_{1},x_{2}\})$ are $k$-cycles with vertices in $\Z_k\times\{0\}$,}\\
& \text{$Dev(\{y_{1},y_{2}\})$ and $Dev(\{y_{3},y_{4}\})$ are $k$-cycles with vertices in $\Z_k\times\{1\}$. }
\end{align*}
Set $A'=(y_{1}, \ov{y_{2}}, x_{3}, \ov{x_{4}}, a_{5},  \ldots, a_{k})$,
$B'=(x_{1}, \ov{x_{2}}, \ov{y_{3}}, y_{4}, b_{5}, \ldots, b_{k})$ and
 let
$\cS = \{\sigma(A'),\ov{\sigma(A')}, \sigma(B'), \ov{\sigma(B')}\}$.
By Lemma \ref{from k-C to k-S}, we have that
$\bigcup_{S\in \cS} Orb(S)$ is a $k$-sun system of
$(\G_2^*[2] + w_2)\setminus (I+\nu)$.

To construct a $k$-sun system of $\G_2^*[2] + w_2$, we build a family $\cT=\{T_1, T_2, T_3, T_4\}$ of
four $k$-suns, each of which is obtained from a graph in $\cS$ by replacing some of their vertices
with $\infty'_1, \infty'_2$, and possibly $\infty'_3$ when $\nu=3$.
Then we construct further
$(2\nu+1)$ $k$-suns $G_1, G_2, \ldots, G_{2\nu +1}$ so that
$\bigcup_{T\in \cT} Orb(T)\ \cup\ \{G_1, G_2, \ldots, G_{2\nu +1}\}$ is a
$k$-sun system of $\G_2^*[2] + w_2$.
\begingroup
\allowdisplaybreaks
\begin{align*}
T_1 &=
\begin{cases}
  \left(
  \begin{matrix}
  y_{1}          &  \ov{y_{2}}    &  x_3   &  \ov{x_{4}} & a_{5}    &  \ldots  & a_{k-1}   & a_{k}\\
  \bm{\infty'_1} & \bm{\infty'_2} &  x_{4} &  \ov{a_5}   & \ov{a_6} &  \ldots  & \ov{a_k}  &  \ov{y_{1}}
  \end{matrix}
  \right) &
  \mbox{if $\nu=2$},\\
  \left(
  \begin{matrix}
  y_{1} &  \ov{y_{2}}  &  \bm{\infty'_3}  &  \ov{x_{4}} & a_{5} &  \ldots  & a_{k-1} & a_{k}\\
\bm{\infty'_1} & \bm{\infty'_2} &  x_{4} & \ov{a_5}   &     \ov{a_6} &  \ldots &\ov{a_k}  &  \ov{y_{1}}
  \end{matrix}
  \right) &
  \mbox{if $\nu=3$},
\end{cases}\\
T_2 &=
  \left(
  \begin{matrix}
  \ov{y_{1}} &  y_{2}  &  \ov{x_{3}}  &  x_{4} & \ov{a_{5}} &  \ldots  & \ov{a_{k-1}} & \ov{a_{k}}\\
\bm{\infty'_1} & \bm{\infty'_2} &  \ov{x_{4}} & a_5   &    a_6 &  \ldots & a_k  &  y_{1}
  \end{matrix}
  \right), \\
T_3 &=
  \left(
  \begin{matrix}
     x_{1} &  \ov{x_{2}} &  \ov{y_{3}}     &  y_{4} &  b_{5} &  \ldots & b_{k-1}  & b_{k}\\
   \bm{\infty'_1}  &      \bm{\infty'_2} & \bm{y_{3}} &     \ov{b_{5}} &       \ov{b_{6}} & \ldots  &    \ov{b_{k}}     & \ov{x_{1}}
  \end{matrix}
  \right),\\
T_4 &=
\begin{cases}
  \left(
  \begin{matrix}
\ov{x_{1}} &  {x_{2}}  &  y_{3} &    \ov{y_{4}} &  \ov{b_{5}}    & \ldots & \ov{b_{k-1}}   & \ov{b_{k}}\\
        \bm{\infty'_1}  &  \bm{\infty'_2}   & y_{4}  & b_{5} &  b_{6} & \ldots & b_{k} & {x_{1}}
  \end{matrix}
  \right) &
  \mbox{if $\nu=2$},\\
  \left(
  \begin{matrix}
\ov{x_{1}} &  {x_{2}}  &  y_{3}  &    \ov{y_{4}} &  \ov{b_{5}}    & \ldots & \ov{b_{k-1}}   & \ov{b_{k}}\\
        \bm{\infty'_1}  & \bm{\infty'_2}    &  \bm{\infty'_3}   & b_5  &  b_{6} & \ldots & b_{k} & {x_{1}}
  \end{matrix}
  \right) &
  \mbox{if $\nu=3$},
\end{cases}
\end{align*}
\endgroup
\[
\begin{array}{ll}
  G_1 = Dev\big(y_{1} \sim y_{2} \sim x_{3} \big),
&
  G_2 = Dev\big(\ov{y_{1}} \sim \ov{y_{2}} \sim \ov{x_{3}} \big),
 \\[1ex] \rule{0pt}{1\normalbaselineskip}
  G_3 = Dev\big(\ov{y_{4}} \sim \ov{y_{3}} \sim x_{2} \big),
&
  G_4 = Dev\big( x_{1} \sim x_{2} \sim \ov{x_{2}} \big),
  \\[1ex] \rule{0pt}{1\normalbaselineskip}
  G_5 =
  \begin{cases}
     Dev\big( \ov{x_{1}} \sim \ov{x_{2}} \sim y_3 \big) & \text{if $\nu=2$},\\
     Dev\big( \{\ov{x_{1}},\ov{x_{2}}\} \oplus \{\ov{x_{4}}, x_3 \} \big) & \text{if $\nu=3$},\\
  \end{cases}
&
  G_6 = Dev\big( y_4 \sim y_3 \sim \ov{x_{2}} \big),
  \\[1ex] \rule{0pt}{1\normalbaselineskip}
  G_7 = Dev\big( x_4 \sim x_3 \sim \ov{y_{2}} \big).
&
\end{array}
\]
By \eqref{thm:holek=3,b2:cond2}, it is not difficult to check that
the graphs $G_h$ are $k$-suns.
\end{proof}

\section{It is sufficient to solve $2k < v < 6k$}
In this section we show that if the necessary conditions in \eqref{nec}, for the existence of
a $k$-sun system of $K_v$, are sufficient for all $v$ satisfying $2k < v < 6k$, then they are sufficient for all $v$. In other words, we prove Theorem \ref{main1}.

We start by showing how to construct $k$-sun systems of $K_{g \times h}$
(i.e., the complete multipartite graph with $g$ parts each of size $h$) when $h=4k$.

\begin{thm}\label{thm:fiber}
For any odd integer $k\geq 3$ and any integer $g\geq 3$,
there exists a $k$-sun system of $K_{g \times 4k}$.
\end{thm}
\begin{proof}
Set $V(K_{g\times 2k})=\Z_{gk}\times [0,1]$ and let
$K_{g\times 4k} = K_{g\times 2k}[2]$.
In \cite[Theorem 2]{HLR} the authors  proved the existence of a $k$-cycle system of $K_{g \times 2k}$.
By applying Lemma \ref{from k-C to k-S} (with $\G = K_{g\times 2k}$ and $u=0$)
we obtain the existence of a $k$-sun system of $K_{g \times 4k}$.
\end{proof}

The following result exploits Theorem \ref{thm:fiber} and shows how to construct
$k$-sun systems of $K_{4kg+n}$, for $g\neq 2$, starting from a $k$-sun system of $K_{4k} + n$ and a $k$-sun system of
either $K_{n}$ or $K_{4k+n}$.

\begin{thm}\label{thm:n+4kg}
Let $k\geq3$ be an odd integer and assume that both the following conditions hold:
\begin{enumerate}
  \item there exists a $k$-sun system of either $K_{n}$ or $K_{4k+n}$;
  \item there exists a $k$-sun system of $K_{4k}+n$.
\end{enumerate}
Then there is a $k$-sun system of  $K_{4kg+n}$ for all positive $g\neq 2$.
\end{thm}
\begin{proof}
Suppose there exists a $k$-sun system $\cS_1$ of  $K_n$, also, by (2), there exists a $k$-sun system $\cS_2$ of $K_{4k}+n$. Clearly, $\cS_1\ \cup\ \cS_2$ is a $k$-sun system of
$K_{n + 4k} = K_{n}\oplus (K_{4k}+n)$. Hence we can suppose $g\geq3$.
Let $V$, $H$ and $G$ be sets of size $n$, $4k$ and $g$, respectively, such that $V\cap (H\times G)=\emptyset$.
Let $\mathcal S$ be a $k$-sun system of $K_{n}$ (resp., $K_{n+4k}$)
with vertex set $V$ (resp., $V \cup (H\times \{x_0\})$ for some $x_0\in G$).
By assumption, for each $x\in G$, there is a $k$-sun system, say ${\mathcal B}_x$,  of $K_{4k}+n$
with vertex set $V\cup (H\times \{x\})$, where $V(K_{4k}) = H\times \{x\}$.
Also, by Theorem \ref{thm:fiber} there is a $k$-sun system $\mathcal C$ of $K_{g \times 4k}$ whose parts are
$H\times \{x\}$ with $x\in G$.
Hence the $k$-suns of ${\mathcal B}_x$ with $x\in G$ (resp., $x \in G\setminus\{x_0\}$), $\mathcal S$ and
$\mathcal C$ form a $k$-sun system of $K_{n+4kg}$ with vertex set $V\cup (H\times G)$.
\end{proof}

We are now ready to prove Theorem \ref{main1} whose statement is recalled below.\\

\noindent
\textbf{Theorem 1.1.}
\emph{
Let $k\geq 3$ be an odd integer and $v>1$. Conjecture \ref{conj} is true if and only if there exists a $k$-sun system of $K_v$ for all $v$ satisfying the necessary conditions in \eqref{nec} with $2k< v <6k$.}
\begin{proof}
The existence of $3$-sun systems and $5$-sun systems has been solved in \cite{FJLS} and in \cite{FHL}, respectively. Hence we can suppose $k\geq 7$ and $2k< v < 6k$.

We first deal with the case where $(k,v)\neq(7,21)$.
By assumption there exists a $k$-sun system of $K_v$,
which implies $v(v-1)\equiv 0 \pmod 4$, hence
Theorem \ref{thm:hole} guarantees the existence of
a $k$-sun system of $K_{4k} + v$.
Therefore, by Theorem \ref{thm:n+4kg} there is
a $k$-sun decomposition of $K_{4kg + v}$ whenever $g\neq 2$.
To decompose $K_{8k + v}$ into $k$-suns, we first decompose
 $K_{8k + v}$ into $K_{4k + v}$ and $K_{4k} + (4k + v)$.
 By Theorem \ref{thm:n+4kg} (with $g=1$), there is a $k$-sun system of $K_{4k + v}$.
 Furthermore, Theorem \ref{thm:hole} guarantees the existence of
 a $k$-sun system of $K_{4k} + (4k + v)$, except possibly when
 $(k,4k + v)\in \{(7,56),(7,57),(7,64),(11,100)\}$.
 Therefore, by Theorem \ref{thm:n+4kg},
 there is a $k$-sun decomposition of $K_{8k + v}$
 whenever $(k,4k + v)\not\in \{(7,56),(7,57),(7,64),(11,100)\}$.
 For each of these four cases we construct $k$-sun systems of $K_{8k+v}$ as follows.

  If $k=7$ and $4k+v=56$, set $V(K_{84})= \Z_{83}\cup\{\infty\}$. We consider the following $7$-suns
  \begin{align*}
  T_1 &=
  \left(
  \begin{matrix}
 0 & -1 & 3 & -4 & 6 & -7 & 16\\
 31 & 27 & 37 & 18 & 43 & 12 & 56
  \end{matrix}
  \right), \\
    T_2 &=
  \left(
  \begin{matrix}
  0 & -2 & 3 & -5 & 6 & -8 & 17\\
  32 & 27 & 38 & 19 & 44 & 12 & 58
  \end{matrix}
  \right), \\
      T_3 &=
  \left(
  \begin{matrix}
 0 & -3 & 3 & -6 & 6 & -9 & 18\\
 33 & 27& 39 & 20 & 45 & 12  & \infty
  \end{matrix}
  \right).
  \end{align*}
  One can easily check that $\bigcup_{i=1}^{3} Orb_{\Z_{83}}(T_i)$  is a $7$-sun system of  $K_{84}$.

  If $k=7$ and $4k+v=57$, set $V(K_{85})= \Z_{85}$. Let $T_1$ and $T_2$ be defined as above, and let $T'_3$ be the graph obtained from $T_3$
  replacing $\infty$ with $60$. It is immediate that $\bigcup_{i=1}^2 Orb_{\Z_{85}}(T_i)\cup Orb_{\Z_{85}}(T'_3)$ is a $7$-sun system of $K_{85}$.

  If $k=7$ and $4k+v=64$, set $V(K_{92})=(\Z_7\times \Z_{13})\cup\{\infty\}$. We consider the following $7$-suns
  \begin{align*}
  T_1 &=
  \left(
  \begin{matrix}
  (0,0) & (1,1) & -(2,1) & (3,1) & -(4,1) & (5,1) & -(6,1)\\
\infty & (-1,1) & (2,7) & (-3,5) & -(3,5) & -(5,7) & (6,7)
  \end{matrix}
  \right), \\
    T_2 &=
  \left(
  \begin{matrix}
  (0,0) & (1,2) & -(2,2) & (3,2) & -(4,2) & (5,2) & -(6,2)\\
(0,10) & -(1,8) & (2,8) & (-3,7) & -(3,7) & -(5,8) & (6,8)
  \end{matrix}
  \right), \\
      T_3 &=
  \left(
  \begin{matrix}
  (0,0) & (1,3) & -(2,3) & (3,3) & -(4,3) & (5,3) & -(6,3)\\
(0,12) & -(1,9) & (2,9) & (-3,9) & -(3,9) & -(5,9) & (6,9)
  \end{matrix}
  \right),
  \end{align*}
 \[
\begin{array}{ll}
  T_4 = Dev_{\Z_7\times\{0\}}\big((0,0) \sim (4,0) \sim (6,8) \big),
&
  T_5 = Dev_{\Z_7\times\{0\}}\big((0,0) \sim (6,0) \sim (6,8) \big).
\end{array}
\]
 One can easily check that
 $\bigcup_{i=1}^{3} Orb_{\Z_7\times \Z_{13}}(T_i)\ \cup\
  \bigcup_{i=4}^{5} Orb_{\{0\}\times \Z_{13}}(T_i)$ is a $7$-sun system of  $K_{92}$.

 If $k=11$ and $4k+v=100$, set $V(K_{144})=(\Z_{11}\times \Z_{13})\cup\{\infty\}$.
 We consider the following $11$-suns
 \begingroup
 \allowdisplaybreaks
 \begin{align*}
 T_1 &=
 \left(
  \begin{matrix}
  (0,0) & (1,1) & -(2,1) & (3,1) & -(4,1) & (5,1) & -(6,1)\\
  \infty & (-1,1) & (2,7) & -(3,7) & (4,7) & (-5,1) & -(5,5)\\
  \end{matrix}\;\;
  \right.\\
  &\hspace{1cm}
  \left.
  \begin{matrix}
  (7,1) & -(8,1) & (9,1) & -(10,1)\\
  -(7,7) & (8,7) & -(9,7) & (10,7)
  \end{matrix}
  \right),
  \\
  T_2 &=
  \left(
  \begin{matrix}
    (0,0) & (1,2) & -(2,2) & (3,2) & -(4,2) & (5,2) & -(6,2)\\
    (0,10) & -(1,8) & (2,8) & -(3,8) & (4,8) & (-5,6) & -(5,7)
  \end{matrix}\;\;
  \right.\\
    &\hspace{1cm}
  \left.
  \begin{matrix}
    (7,2) & -(8,2) & (9,2) & -(10,2)\\
    -(7,8) & (8,8) & -(9,8) & (10,8)
  \end{matrix}
  \right), \\
  T_3 &=
  \left(
  \begin{matrix}
   (0,0) & (1,3) & -(2,3) & (3,3) & -(4,3) & (5,3) & -(6,3) \\
  (0,12) & -(1,9) & (2,9) & -(3,9) & (4,9) & (-5,9) & -(5,9)
  \end{matrix}\;\;
  \right.\\
    &\hspace{1cm}
  \left.
  \begin{matrix}
    (7,3) & -(8,3) & (9,3) & -(10,3)\\
   -(7,9) & (8,9)  & -(9,9) & (10,9)
  \end{matrix}
  \right),
\end{align*}
\endgroup
 \[
\begin{array}{ll}
  T_4 = Dev_{\Z_{11}\times\{0\}}\big((0,0) \sim (4,0) \sim (6,8) \big),
&
  T_5 = Dev_{\Z_{11}\times\{0\}}\big((0,0) \sim (6,0) \sim (5,8) \big),
  \\[1ex] \rule{0pt}{1\normalbaselineskip}
  T_6 = Dev_{\Z_{11}\times\{0\}}\big((0,0) \sim (8,0) \sim (8,8) \big). &
\end{array}
\]
One can check that
$\bigcup_{i=1}^{3} Orb_{\Z_{11}\times \Z_{13}}(T_i) \ \cup\
\bigcup_{i=4}^{6} Orb_{\{0\}\times \Z_{13}}(T_i)$
is an $11$-sun system of  $K_{144}$.

It is left to prove the
existence of a $k$-sun system of $K_{4kg+v}$ when $(k,v)=(7,21)$ and for every $g\geq 1$.
If $g=1$, a $7$-sun system of $K_{49}$ can be obtained as a particular case of the following construction. Let $p$ be a prime, $q=p^n\equiv1\pmod 4$ and $r$ be a primitive root of $\mathbb{F}_q$.
Setting
$S=Dev_{\langle r\rangle}(0 \sim r\sim r+1)$ where $\langle r\rangle=\{jr\mid 1\leq j\leq p\}$,
we have that $\bigcup_{i=0}^{\frac{q-5}{4}} Orb_{\mathbb{F}_q} (r^{2i}S)$ is a $p$-sun system of $K_q$.

If $g\geq2$, we notice that  $K_{28g + 21} = K_{28(g-1)+49}$. Considering the
$7$-sun system of $K_{49}$ just built, and recalling that by Theorem \ref{thm:hole}
there is a $7$-sun system of $K_{28}+49$, then Theorem \ref{thm:n+4kg} guarantees the existence
of a $7$-sun system of $K_{28(g-1)+49}$ whenever $g\neq 3$.
When $g=3$, a $7$-sun system of $K_{105}$ is constructed as follows.
Set $V(K_{105})=\Z_7\times \Z_{15}$.
 Let
 $S_{i,j}$ and $T$ be the $7$-suns defined below, where
 $(i,j)\in X= ([1,3]\times [1,7])\setminus\{(1,3), (1,6)\}$:
  \begin{align*}
  S_{i,j} &=
  \left(
  \begin{matrix}
  (0,0) & (i,j/2) & (2i,j)  & (3i,0)  & (4i,j)  & (5i,0) & (6i,j)\\
(i,-j/2) & (2i,0) & (3i,2j) & (4i,-j) & (5i,2j) & (6i,-j)& (0,2j )
  \end{matrix}
  \right),\\
  T &=
  \left(
  \begin{matrix}
  (0,0) & (0,7) & (0,2)  & (0,5)  & (0,-1)  & (0,3) & (0,1)\\
  (2,0) & (3,7) & (1,2)  & (1,8)  & (1,5)   & (1,0) & (1,10)
  \end{matrix}
  \right).
  \end{align*}
   One can check that
   $\displaystyle \bigcup_{(i,j)\in X} Orb_{\{0\}\times \Z_{15}}(S_{i,j})
   \ \cup\
   Orb_{\Z_{7}\times \Z_{15}}(T)$ is a $7$-sun system of $K_{105}$.
\end{proof}

\section{Construction of $p$-sun systems, $p$ prime}
In this section we prove Theorem \ref{main2}.
Clearly in view of Theorem \ref{main1} it is sufficient to construct a $p$-sun system of $K_v$ for any admissible $v$ with $2p < v< 6p$.
Hence, we are going to prove the following result.
\begin{thm}\label{pprimevpiccolo}
Let $p$ be an odd prime and let $v(v-1)\equiv 0\pmod{4p}$ with $2p< v< 6p$. Then there exists a $p$-sun system of  $K_v$.
\end{thm}
Since the existence of $p$-sun systems with $p=3,5$ has been proved in \cite{FJLS} and in \cite{FHL}, respectively, here we can assume $p\geq 7$.

It is immediate to see that by the necessary conditions for the existence of a $p$-sun system of  $K_v$,
 it follows that $v$ lies in one of the following congruence classes modulo $4p$:
\begin{itemize}
\item [1)] $v\equiv 0,1\pmod{4p}$;
\item [2)] $v\equiv p,3p+1\pmod{4p}$ if $p\equiv 1\pmod 4$;
\item [3)] $v\equiv p+1,3p\pmod{4p}$  if $p\equiv 3\pmod 4$.
\end{itemize}

If $v\equiv 0,1\pmod{4p}$ we present a direct construction which holds more in general for $p=k$, where
$k$ is an odd integer and not necessarily a prime.
\begin{thm}\label{v=01mod4k}
For any $k=2t+1\geq7$  there exists a  $k$-sun system of $K_{4k+1}$
and a $k$-sun system of  $K_{4k}$.
\end{thm}
\begin{proof}
Let $C$ be the $k$-cycle with vertices in $\Z$ so defined:
$$C=(0,-1,1,-2,2,-3,3,\ldots,1-t, t-1,-t,2t).$$
Note that the list $D_1$ of the positive differences in $\Z$ of $C$ is $D_1=[1,2t]\cup \{3t\}$.
Consider now the ordered $k$-set $D_2=\{d_{1},d_{2},\ldots,d_{k}\}$
so defined:
$$D_2=[2t+1,3t-1]\cup [3t+1,4t+2].$$
Obviously  $D_1 \cup D_2=[1,2k].$
Let $\{c_{1},c_{2},\ldots,c_{k}\}$
be the increasing order of the vertices of the cycle $C$
and set $\ell_{r}=c_{r}+d_{r}$ for every $r\in[1,k]$,
with $r\neq \frac{t+1}{2}$, and
$\ell_{\frac{t+1}{2}}=c_{\frac{t+1}{2}}-d_{\frac{t+1}{2}}$ when $t$ is odd.
It is not hard to see that $V=\{c_{1},c_{2},\ldots,c_{k},\ell_{1},\ell_{2},\ldots,\ell_{k}\}$ is a set.
Note also that $V\subseteq \{-3t-1\}\cup [-t,5t] \cup\{6t+2\}$.

Let $S$ be the sun obtainable from $C$ by adding the pendant edges $\{c_{i},\ell_{i}\}$
for $i\in[1, k]$.
Clearly, $\Delta S=\pm (D_1\ \cup\ D_2)=\pm [1,2k]$.
So we can conclude that if we consider the vertices of $S$ as elements of $\Z_{4k+1}$,
the vertices are still pairwise distinct and $\Delta S=\Z_{4k+1}\setminus\{0\}$.
Then, by applying Corollary \ref{mixed diff method for Kmn} (with $G=\Z_{4k+1}, m=1, w=0$),
it follows that
$Orb_{\Z_{4k+1}}S$ is a $k$-sun system of  $K_{4k+1}$.

Now we construct a  $k$-sun system of $K_{4k}$.
 Let $S$ be defined as above and
note that $d_{k}=2k$. Let $S^*$
be the sun obtained by $S$ setting $\ell_{k}=\infty$.
It is immediate that if we consider the vertices of $S^*$ as elements of $\Z_{4k-1}\cup \{\infty\}$,
then Corollary \ref{mixed diff method for Kmn} (with $G=\Z_{4k-1}, m=1, w=1$) guarantees that $Orb_{\Z_{4k-1}}S^*$ is a $k$-sun system of  $K_{4k}$.
\end{proof}

\begin{ex}
Let $k=2t+1=9$, hence $t=4$.
By following the proof of Theorem \ref{v=01mod4k}, we construct a $9$-sun system of  $K_{37}$.
Taking $C=(0,-1,1,-2,2,-3,3,-4,8)$, we have that
  \begin{eqnarray*}
    \{d_1,d_2,\ldots,d_9\} &=& [9,11]\ \cup\ [13,18]\\
    \{c_1,c_2,\ldots,c_9\} &=& \{-4,-3,-2,-1,0,1,2,3,8\}.
  \end{eqnarray*}
Hence $\{\ell_1,\ell_2,\ldots,\ell_9\}=\{5,7,9,12,14,16,18,20,26\}$
and we obtain the following $9$-sun $S$ with vertices in $\Z_{37}$:
\begin{equation}\nonumber
 S= \left(
  \begin{array}{ccccccccc}
    0  & -1      & 1 & -2 & 2 & -3 & 3 & -4 & 8   \\
    14 & 12 & 16 & 9 & 18 & 7 & 20 & 5 & 26
  \end{array}
  \right),
\end{equation}
such that $\Delta S=\Z_{37}\setminus \{0\}$.
Therefore, $Orb_{\Z_{37}}S$ is a $9$-sun system of  $K_{37}$.
\end{ex}

From now on,
we assume that $p$ is an odd prime number and denote by $\Sigma$ the following $p$-sun:

\begin{equation}\nonumber
 \Sigma= \left(
  \begin{array}{ccccc}
    c_0  & c_1      & \ldots & c_{p-2}   & c_{p-1} \\
    \ell_0 & \ell_1     & \ldots & \ell_{p-2}  & \ell_{p-1}
  \end{array}
  \right).
\end{equation}

\begin{lem}\label{lemmaij}
Let $p$ be an odd prime. For any $x,y\in\Z_p$ with $x\neq0$ and any $i,j\in \Z_m$ with $i\neq j$
there exists a $p$-sun $S$ such that $\Delta_{ii}S=\pm{x}$, $\Delta_{ij}S=y$, $\Delta_{ji}S=-y$ and $\Delta_{hk}S=\emptyset$ for any $(h,k)\in (\Z_m\times \Z_m)\setminus\{(i,i),(i,j),(j,i)\}$.
\end{lem}
\begin{proof}
It is easy to see that $S=Dev_{\Z_p\times\{0\}}((0,i)\sim (x,i)\sim (y+x,j))$ is the required $p$-sun.
\end{proof}

We will call such a $p$-sun a \emph{sun of type $(i,j)$}.
For the following it is important to note that if $S$ is a $p$-sun of type $(i,j)$,
then $|\Delta_{ii}S|=2$, $|\Delta_{jj}S|=0$ and $|\Delta_{ij}S|=|\Delta_{ji}S|=1$.

The following two propositions provide us $p$-sun systems of $K_{mp+1}$ whenever $m\in\{3,5\}$
and $p\equiv m-2 \pmod{4}$.

\begin{prop}\label{v=3p+1}
Let $p\equiv 1 \pmod 4\geq 13$ be a prime. Then  there exists a  $p$-sun system of  $K_{3p+1}$.
\end{prop}
\begin{proof}

We have to distinguish two cases according to the congruence of $p$ modulo 12.\\
Case 1. Let $p\equiv 1\pmod{12}$.

If $p=13$, we construct a $13$-sun system of $K_{40}$ as follows.
Let $S$ be the following $13$-sun whose vertices are labelled with elements of $(\Z_{13}\times \Z_3)\cup\{\infty\}$:
\begingroup
\allowdisplaybreaks
\begin{align*}
S&= \left(
  \begin{matrix}
  \infty & (2,1) & (4,2) & (8,0) & (3,1) & (6,2) & (12,0) \\
  (0,2) & (4,1) & (8,1) & (3,2) & (6,0) & (12,1) & (11,2) \\
  \end{matrix}\;\;
  \right.\\
  &\hspace{1cm}
  \left.
  \begin{matrix}
  (11,1) & (9,2) & (5,0) & (10,1) & (7,2) & (1,0)\\
  (9,0) & (5,1) & (10,2) & (7,0) & (1,1) & (2,2)
   \end{matrix}
  \right).
\end{align*}
\endgroup
We have:
\begin{align*}
\Delta_{12}S&=\Delta_{21}S=\pm\{2,3,4,6\}, \quad\quad \Delta_{02}S=\Delta_{20}S=\pm\{1,4,5,6\},\\
\Delta_{01}S&=-\Delta_{10}S=
\{-1,2,\pm 3,\pm 5\},\quad\quad
\Delta_{00}S=\Delta_{22}S=\emptyset,\quad\quad \Delta_{11}S=\pm\{2\}.
\end{align*}
Now it remains to construct a set $\cT$ of edge-disjoint $13$-suns such that
\begin{align*}
\Delta_{12}{\cT}&=\Delta_{21}{\cT}=\{0,\pm 1,\pm 5\}, \quad\quad \Delta_{02}{\cT}=\Delta_{20}{\cT}=\{0,\pm2,\pm3\},\\
\Delta_{01}{\cT}&=-\Delta_{10}{\cT}= \{0,1,-2,\pm 4,\pm 6\},\ \
\Delta_{00}{\cT}=\Delta_{22}{\cT}=\Z_{13}^*,\ \
 \Delta_{11}{\cT}=\Z_{13}^*\setminus\{\pm 2\}.
\end{align*}
In order to do this it is sufficient to take, $\cT=\{T_{01}^i \mid i\in[1,4]\}\cup\{T_{02}^i\mid i\in[1,2]\}\cup\{T_{10}^i\mid i\in[1,3]\}
\cup\{T_{12}^i \mid i\in[1,2]\}\cup \{T_{20}^i \mid i\in[1,3]\}\cup \{T_{21}^i \mid i\in[1,3]\}$, where:
\begin{align*}
T_{01}^i&=Dev_{\Z_{13}\times\{0\}}((0,0)\sim (x_i,0)\sim (y_i+x_i,1)), \textrm{where}\ x_i\in[1,4],\ y_i\in \pm\{4,6\},\\
T_{02}^i&=Dev_{\Z_{13}\times\{0\}}((0,0)\sim (x_i,0)\sim (y_i+x_i,2)), \textrm{where}\ x_i\in[5,6],\ y_i\in \pm\{2\},\\
T_{10}^i&=Dev_{\Z_{13}\times\{0\}}((0,1)\sim (x_i,1)\sim (y_i+x_i,0)), \textrm{where}\ x_i\in\{1,3,4\},\ y_i\in \{0,-1,2\},\\
T_{12}^i&=Dev_{\Z_{13}\times\{0\}}((0,1)\sim (x_i,1)\sim (y_i+x_i,2)), \textrm{where}\ x_i\in[5,6],\ y_i\in \pm\{1\},\\
T_{20}^i&=Dev_{\Z_{13}\times\{0\}}((0,2)\sim (x_i,2)\sim (y_i+x_i,0)), \textrm{where}\ x_i\in[1,3],\ y_i\in \{0,\pm3\},\\
T_{21}^i&=Dev_{\Z_{13}\times\{0\}}((0,2)\sim (x_i,2)\sim (y_i+x_i,1)), \textrm{where}\ x_i\in[4,6],\ y_i\in \{0,\pm5\}.
\end{align*}
We have that $\cT\cup Orb_{\Z_{13}\times\{0\}}S$ is a $13$-sun system of $K_{40}$.

Suppose now that $p\geq37$. We proceed in a very similar way to the previous case.
Let $r$ be a primitive root of $\Z_p$.
Consider the $((\Z_{p}\times \Z_{3})\cup\{\infty\})$-labeling $B$ of $\Sigma$ so defined:
\begin{align*}
B(c_0)=\infty;&\quad B(c_i)=(r^i,i)\quad {\rm for} \ 1\leq i\leq p-1 \\
B(\ell_0)=(0,2);&\quad B(\ell_i)=(r^{i+1},i+2)
    \end{align*}
except for $\frac{p-9}{4}$ values of $i \equiv 1 \pmod 3$ for which we set $B(\ell_i)=(r^{i+1},i)$.
Letting $S=B(\Sigma)$,
it is immediate that the labels of the vertices of $S$ are pairwise distinct.
Note that
$$|\Delta_{00}S|=|\Delta_{22}S|=0, \quad |\Delta_{11}S|=\frac{p-9}{2},\quad |\Delta_{01}S|=|\Delta_{10}S|=\frac{5p+7}{12},$$
$$|\Delta_{ij}S|=\frac{2p-2}{3}\ {\rm for} \ (i,j)\in\{(0,2),(1,2),(2,0),(2,1)\}.$$
Hence, reasoning as in the previous case, we have to construct a set $\cT$ of $p$-suns
such that if $i\neq j$ then $\Delta_{ij}\cT=\Z_p\setminus \Delta_{ij}S$ is a set and also
$\Delta_{ii}\cT=\Z^*_p\setminus \Delta_{ii}S$ is a set.
In particular, this implies that for any $T,T' \in \cT$ we have $\Delta_{ij}T \cap \Delta_{ij}T'=\emptyset$
and that $|\Delta_{00}\cT|=|\Delta_{22}\cT|=p-1$,
$|\Delta_{11}\cT|=\frac{p+7}{2}$,
$|\Delta_{ij}\cT|=\frac{p+2}{3}$ for $(i,j)\in\{(0,2),(1,2),(2,0),(2,1)\}$,
and $|\Delta_{01}\cT|=|\Delta_{10}\cT|=\frac{7p-7}{12}$.
In order to do this it is sufficient to take $\cT$ as a set consisting of
$\frac{p-1}{2}$ suns of type $(0,1)$,
$\frac{p-1}{12}$ suns of type $(1,0)$,
$\frac{p+11}{6}$ suns of type $(1,2)$,
$\frac{p+2}{3}$ suns of type $(2,0)$,
$\frac{p-7}{6}$ suns of type $(2,1)$,
which exist in view of Lemma \ref{lemmaij}.
We have that $Orb_{\Z_p\times\{0\}}S \cup \cT$ is a $p$-sun system of  $K_{3p+1}$.

Case 2. Let $p\equiv 5\pmod{12}$.  Let $r$ be a primitive root of $\Z_p$.
Consider the $((\Z_{p}\times \Z_{3})\cup\{\infty\})$-labeling $B$ of $\Sigma$ so defined:
\begin{align*}
B(c_0)=&\infty;\quad B(c_i)=(r^i,i)\quad {\rm for} \ 1\leq i\leq p-2;\quad B(c_{p-1})=(1,0);\\
B(\ell_0)=&(0,2);\quad B(\ell_1)=(r,2);\quad B(\ell_i)=
  \begin{cases}
(r^{i-1},i+1)& {\rm for}\ i\in \left[2,\frac{p-1}{2}\right]\\
(r^{i+1},i+2)& {\rm for}\ i\in \left[\frac{p+1}{2},p-3\right]\\
\end{cases}\\
B(\ell_{p-2})=&(1,1);\quad B(\ell_{p-1})=(1,2);
\end{align*}
except for $\frac{p-17}{6}$ values of $i\equiv 0\pmod 3$ with $i\in\left[3,\frac{p-1}{2}\right]$ for which we set
$B(\ell_i)=(r^{i-1},i)$ and
$\frac{p-5}{12}$ values of $i\equiv 0\pmod 3$ with $i\in\left[\frac{p+1}{2},p-5\right]$ for which we set
$B(\ell_i)=(r^{i+1},i)$.
Letting $S=B(\Sigma)$, it is easy to see that the labels of the vertices of $S$ are pairwise distinct.
Note that
\begin{align*}
|\Delta_{00}S|&=\frac{p-9}{2},\quad |\Delta_{11}S|=|\Delta_{22}S|=0, \quad
 |\Delta_{01}S|=|\Delta_{10}S|=\frac{p+1}{2},\\
  |\Delta_{02}S|&=|\Delta_{20}S|=\frac{7p+1}{12},\quad |\Delta_{12}S|=|\Delta_{21}S|=\frac{2p-4}{3}.
 \end{align*}
Hence, we have to construct a set $\cT$ of $p$-suns
such that $|\Delta_{11}\cT|=|\Delta_{22}\cT|=p-1$,
$|\Delta_{00}\cT|=\frac{p+7}{2}$,
$|\Delta_{01}\cT|=|\Delta_{10}\cT|=\frac{p-1}{2}$,
$|\Delta_{02}\cT|=|\Delta_{20}\cT|=\frac{5p-1}{12}$,  and
$|\Delta_{12}\cT|=|\Delta_{21}\cT|=\frac{p+4}{3}$.
 In order to do this it is sufficient to take $\cT$ as a set consisting of
$\frac{p+7}{4}$ suns of type $(0,1)$,
$\frac{p-9}{4}$ suns of type $(1,0)$,
$\frac{p+7}{4}$ suns of type $(1,2)$,
$\frac{5p-1}{12}$ suns of type $(2,0)$, and
$\frac{p-5}{12}$ suns of type $(2,1)$
which exist in view of Lemma \ref{lemmaij}.
We have that $Orb_{\Z_p}S \cup \cT$ is a $p$-sun system of $K_{3p+1}$.
\end{proof}

\begin{prop}\label{case2d}
For any prime $p\equiv 3\pmod 4$ there exists a $p$-sun system of $K_{5p+1}$.
\end{prop}
\begin{proof}
Set $p=4n+3$, and let $Y=[1,n]$ and $X=[n+1,2n+1]$.
%
Consider the following $(\Z_p\times \Z_{5})\cup\{\infty\}$-labeling $B$ of $\Sigma$ defined as follows:
\begin{align*}
& B(c_0)=(0,0) ;\quad B(c_i)=(-1)^{i+1}(i,1)\quad {\rm for\ every} \ i\in[1,p-1];\\
& B(\ell_0)=\infty;\quad\, \quad B(\ell_y)=(-1)^y(y,-1)\quad  {\rm for\ every \ } y\in Y;\\
& B(\ell_{2n+1})=(-2n-1,3);\quad B(\ell_{2n+2})=(-2n-1,-3); \\
& B(\ell_i)=(-1)^i(i,3)\quad {\rm for\ every} \ i\in[1,p-1]\setminus(Y\cup \{2n+1,2n+2\}).
\end{align*}
One can directly check that the vertices of $S=B(\Sigma)$ are pairwise distinct. Also,
it is not hard to verify that $\Delta S$ does not have repetitions
and that its complement in $(\Z_p\times \Z_{5})\setminus\{(0,0)\}$ is the set
$$D=\{\pm(2x,0) \ | \ x\in X\} \ \cup \ \{\pm(2y,4) \ | \ y\in Y\} \ \cup \ \{\pm(0,1)\}.$$

Clearly, $D$ can be partitioned into $n+1$ quadruples of the form
$D_x=\{\pm(2x,0),$ $\pm(r_x,s_x)\}$ with $x\in X$  and $s_x\neq 0$.
Letting
\[S_x = Dev_{\Z_p\times \{0\}}\big((0,0)\sim (2x,0) \sim (r_x+2x,s_x)\big)
\]
for $x\in X$,
it is clear that $\Delta S_x = D_x$, hence $\Delta \{S_x\mid x\in X\} = D$.
Therefore, Corollary \ref{mixed diff method for Kmn} guarantees that
$\bigcup_{x\in X} Orb_{\{0\}\times \Z_5} (S_x)\ \cup\ Orb_{\Z_p\times \Z_{5}}(S)$
is a $p$-sun system of $K_{5p+1}$.
\end{proof}

\begin{ex}
 Here, we construct a $7$-sun system of $K_{36}$ following the proof of Proposition \ref{case2d}.
 In this case, $Y=\{1\}$ and $X=\{2,3\}$. Now consider the $7$-sun $S$ defined below, whose vertices lie in $(\Z_7\times \Z_{5})\cup\{\infty\}$:
{\begin{align*}
S= \left(
  \begin{matrix}
  (0,0) & (1,1) &  -(2,1) &  (3,1) &  -(4,1) &  (5,1) & -(6,1) \\
  \infty & -(1,-1) & (2,3) & (-3,3) & -(3,3) & -(5,3) & (6,3)
  \end{matrix}
  \right).
\end{align*}}
We have
\begin{eqnarray*}
  \Delta S & = & \pm\big\{(1,1),(3,2),(5,2),(0,2),(2,2),(4,2),(6,1),(2,0),(4,4),(6,-2),\\
  & & (1,-2),(3,4),(5,4)\big\}.
\end{eqnarray*}
Hence $\Delta S$ does not have repetitions
and its complement in $(\Z_7\times \Z_{5})\setminus\{(0,0)\}$ is the set
$$D=\pm\{(4,0),(6,0),(2,4),(0,1)\}.$$
Now it is sufficient to take 
\[
S_2=Dev_{\Z_7\times\{0\}}((0,0) \sim (4,0) \sim (6,4)) \quad
S_3=Dev_{\Z_7\times\{0\}}((0,0) \sim (6,0) \sim (6,1)).
\]
One can check that
$\bigcup _{x\in X} Orb_{\{0\}\times \Z_5} (S_x)\ \cup\ Orb_{\Z_7\times \Z_{5}}S$
is a $7$-sun system of $K_{36}$.
\end{ex}

We finally construct $p$-sun systems of $K_{mp}$ whenever $p\equiv m \pmod{4}$.
\begin{prop}\label{v=3p, 5p}
Let $m$ and $p$ be odd prime numbers
with $m\leq p$ and $m\equiv p\pmod{4}$. Then there exists a $p$-sun system of $K_{mp}$.
\end{prop}
\begin{proof}
For each pair $(r,s)\in \Z^*_{p} \times \Z_{m}$, let $B_{r,s}:V(\Sigma)\rightarrow \Z_p\times \Z_{m}$
be the labeling of the vertices of $\Sigma$ defined as follows:
\begin{align*}
  & B_{r,s}(c_0) = (0,0), \\
  & B_{r,s}(c_i) = B_{r,s}(c_{i-1}) +
    \begin{cases}
      (r, s) & \text{if $i\in [1, m+1] \cup\{m+3,m+5, \ldots, p-1\}$},\\
      (r, -s)& \text{if $i\in \{m+2, m+4, \ldots, p-2\}$},
    \end{cases} \\
  & B_{r,s}(\ell_{i}) = B_{r,s}(c_{i}) +
    \begin{cases}
      (r, -s) & \text{if $i\in [0, m] \cup\{m+2,m+4, \ldots, p-2\}$},\\
      (r, s)  & \text{if $i\in \{m+1,m+3, \ldots, p-1\}$}.
    \end{cases}
\end{align*}
Since $B_{r,s}$ is injective,
for every $h\in \Z_{m}$ the graph $S_{r,s}^h = \tau_{(0,h)}\big(B_{r,s}(\Sigma)\big)$ is a $p$-sun.
For $i,j\in \Z_{m}$, we also notice that
$\Delta_{ij} \{S_{r,s}^h \mid h\in \Z_{m}\} =\{\pm~r\}$ whenever
$i-j=\pm s$, otherwise it is empty.

Letting $\mathcal{S}$ be the union of the following two sets of $p$-suns:
\begin{align*}
  & \{S^h_{r,1} \mid h\in\Z_{m}, r\in \left[1, (p+m-2)/4\right]\},\\
  & \{S^h_{r,s} \mid h\in\Z_{m}, r\in \left[1, (p-1)/2\right], s\in[2, (m-1)/2]\},
\end{align*}
it is not difficult to see that for every $i,j\in\Z_{m}$
\[
 \Delta_{ij} \mathcal{S}=
 \begin{cases}
   \varnothing                       & \text{if $i=j$},\\
   \pm\left[ 1, \frac{p+m-2}{4}\right] & \text{if $i-j = \pm1 $}, \\
   \Z^*_{p}                          & \text{otherwise}.
 \end{cases}
\]

It is left to construct a set $\cT$ of $p$-suns
such that $\Delta_{ij}\cT=\Z_p\setminus \Delta_{ij}\mathcal{S}$ whenever $i\neq j$,
and $\Delta_{ii}\cT=\Z^*_p \setminus \Delta_{ii}\mathcal{S} = \Z^*_p$. Therefore,
\[
 |\Delta_{ij} \mathcal{T}| =
 \begin{cases}
   p-1                      & \text{if $i=j$},\\
   \frac{p-m}{2}+1          & \text{if $i-j = \pm1$}, \\
   1                        & \text{otherwise}.
 \end{cases}
\]
It is enough to take $\cT$ as a set consisting of
one sun of type $(h, h+x)$
and $\frac{p-m}{2}$ suns of type $(h,h+1)$,
for every $h\in\Z_{m}$ and $x\in[1, \frac{m-1}{2}]$.
These $p$-suns exist by Lemma \ref{lemmaij}, therefore
$\mathcal{S} \cup \cT$ is the desired $p$-sun system of  $K_{mp}$.
\end{proof}

\begin{ex}
  Let $(m,p)=(3,11)$. Following the proof of Proposition \ref{v=3p, 5p}, we construct
  an $11$-sun system of $K_{33}$.  For every $h\in\Z_{3}$ and $r\in [1,3]$, let $S^h_{r,1}$ be the $11$-sun
  defined below:
\begin{align*}
S^h_{r,1}&=
 \left(
  \begin{matrix}
  (0,h)  &  (r,h+1) & (2r,h+2) & (3r, h)     & (4r,h+1) & (5r,h)  \\
  (r,h+2)  & (2r,h) & (3r,h+1) & (4r, h+2)     & (5r,h+2) & (6r,h+2) \\
 \end{matrix}\;\;
  \right.\\
  &\hspace{1cm}
  \left.
  \begin{matrix}
  (6r,h+1) & (7r, h)   & (8r,h+1) &  (9r,h)    & (10r,h+1) \\
  (7r,h+2) & (8r, h+2) & (9r,h+2) & (10r,h+2)  & (0,h+2)
    \end{matrix}
  \right).
\end{align*}
  One can check that $\Delta_{ij}\{S^0_{r,1}, S^1_{r,1}, S^2_{r,1}\} = \{\pm r\}$ if $i\neq j$,
  otherwise it is empty. Therefore,
  letting $\mathcal{S} = \{S^h_{r,1}\mid h\in \Z_{3}, r\in[1,3]\}$,
  we have that $\Delta_{ij} \mathcal{S}$ is non-empty only when $i\neq j$, in which case
  we have  $\Delta_{ij} \mathcal{S}= \pm[1,3]$.

  Now let $\cT=\{T_{hg} \mid h\in\Z_{3}, g\in[1,5]\}$ where $T_{hg}$ is the $11$-sun defined as follows:
\begin{align*}
  T_{h1} &= Dev_{\Z_{11}\times\{0\}}((0,h)\sim (1,h)\sim (1,h+1)),\\
  T_{hg} &= Dev_{\Z_{11}\times\{0\}}((0,h)\sim (g,h)\sim (9,h+1)), \;\;\text{for every $g\in [2,5]$}.
\end{align*}
Note that each $T_{hg}$ is an $11$-sun of type $(h, h+1)$. Therefore we have that
\[
\Delta_{ij} \mathcal{T} =
\begin{cases}
  \pm [1,5]          & \text{if $0\leq i=j \leq 2$},\\
  \{0\}\ \cup\ [4,7] & \text{otherwise}.\\
\end{cases}
\]
By Corollary \ref{mixed diff method for Kmn}, it follows that $\cS\cup \cT$
is an $11$-sun system of $K_{33}$.
\end{ex}

We are now ready to show that the necessary conditions for the existence of  a $p$-sun system of
$K_v$ are also sufficient whenever $p$ is an odd prime. In other words,
we end this section by proving Theorem \ref{pprimevpiccolo}.\\

\noindent
\emph{Proof of Theorem \ref{pprimevpiccolo}}.
  If $p=3,5$ the result can be found in \cite{FJLS} and in \cite{FHL}, respectively.
  For $p\geq 7$, the result follows from Propositions \ref{v=3p+1}, \ref{case2d} and \ref{v=3p, 5p}.

\section*{Acknowledgements}

The authors gratefully acknowledge support from GNSAGA of Istituto Nazionale di Alta Matematica.

\end{document}